\documentclass[a4paper,12pt]{amsart}

\usepackage[utf8]{inputenc}
\usepackage[english]{babel}
\usepackage{amsfonts}
\usepackage{amsmath}
\usepackage{mathrsfs}
\usepackage{amssymb}
\usepackage{amsthm}
\usepackage{amscd}
\usepackage{latexsym}
\usepackage{amstext}
\usepackage{amsxtra}
\usepackage[alphabetic
,lite]{amsrefs} % for bibliography
\usepackage{color}
\usepackage{paralist}
\usepackage[colorinlistoftodos]{todonotes}
\usepackage[all]{xy}
\usepackage{enumerate}
\usepackage{multirow}
\usepackage{wasysym}
\usepackage{latexsym} % for \Box
\usepackage{comment}
\usepackage{colonequals}
\usepackage{tikz}
\usetikzlibrary{shapes,external}
\usepackage{lipsum}
\tikzset{naming/.style={align=center,font=\footnotesize}}
\tikzset{area/.style = {draw, shape = regular polygon, regular polygon sides = 6, thick, minimum width = 5cm}}
\usepackage{graphicx} % Required for including images
\usepackage[font=small,labelfont=bf]{caption} % Required for specifying captions to tables and figures

\usepackage{geometry}
\usepackage[
        draft=false,
        colorlinks, citecolor=darkgreen,
        pdfauthor={F. F. Favale, J.C. Naranjo, G. P. Pirola, S. Torelli},
        pdftitle={Coverings},
        linktocpage        
]{hyperref}
\hypersetup{citecolor=blue,linktocpage}

\newtheorem{theorem}{Theorem}[section]
\newtheorem*{theorem*}{Theorem}
\newtheorem{lemma}[theorem]{Lemma}
\newtheorem{corollary}[theorem]{Corollary}
\newtheorem{proposition}[theorem]{Proposition}
\newtheorem{remark}[theorem]{Remark}
\newtheorem{definition}[theorem]{Definition}
 
\newtheorem{conj}[theorem]{Conjecture}

\newtheorem{conjecture}{Conjecture}

\newcommand{\nc}{\newcommand}

\nc{\cH}{{\mathcal H}}
\nc{\cA}{{\mathcal A}}
\nc{\cG}{{\mathcal G}}
\nc{\cC}{{\mathcal C}}
\nc{\cD}{{\mathcal D}}
\nc{\cO}{{\mathcal O}}
\nc{\cI}{{\mathcal I}}
\nc{\cB}{{\mathcal B}}
\nc{\cY}{{\mathcal Y}}
\nc{\cK}{{\mathcal K}} 
\nc{\cX}{{\mathcal X}}
\nc{\cS}{{\mathcal S}}
\nc{\cE}{{\mathcal E}}
\nc{\cF}{{\mathcal F}}
\nc{\cZ}{{\mathcal Z}}
\nc{\cQ}{{\mathcal Q}}
\nc{\cN}{{\mathcal N}}
\nc{\cP}{{\mathcal P}}
\nc{\cL}{{\mathcal L}}
\nc{\cM}{{\mathcal M}}
\nc{\cT}{{\mathcal T}}
\nc{\cW}{{\mathcal W}}
\nc{\cU}{{\mathcal U}}
\nc{\cJ}{{\mathcal J}}
\nc{\cV}{{\mathcal V}}
\nc{\cR}{{\mathcal R}}
\nc{\bH}{{\mathbb H}}
\nc{\bA}{{\mathbb A}}
\nc{\bG}{{\mathbb G}}
\nc{\bC}{{\mathbb C}}
\nc{\bO}{{\mathbb O}}
\nc{\bI}{{\mathbb I}}
\nc{\bB}{{\mathbb B}}
\nc{\bY}{{\mathbb Y}}
\nc{\bK}{{\mathbb K}} 
\nc{\bX}{{\mathbb X}}
\nc{\bS}{{\mathbb S}}
\nc{\bE}{{\mathbb E}}
\nc{\bF}{{\mathbb F}}
\nc{\bZ}{{\mathbb Z}}
\nc{\bQ}{{\mathbb Q}}
\nc{\bN}{{\mathbb N}}
\nc{\bP}{{\mathbb P}}
\nc{\bL}{{\mathbb L}}
\nc{\bM}{{\mathbb M}}
\nc{\bT}{{\mathbb T}}
\nc{\bW}{{\mathbb W}}
\nc{\bU}{{\mathbb U}}
\nc{\bD}{{\mathbb D}}
\nc{\bJ}{{\mathbb J}}
\nc{\bV}{{\mathbb V}}
\nc{\bR}{{\mathbb R}}

\nc{\OO}{\mathcal{O}}
\nc{\PP}{\mathbb{P}}

\DeclareMathOperator{\Spec}{Spec}

\DeclareMathOperator{\Pic}{Pic}

\DeclareMathOperator{\Alb}{Alb}
\DeclareMathOperator{\homo}{hom}
\DeclareMathOperator{\alb}{alb}
\DeclareMathOperator{\cl}{cl}
\DeclareMathOperator{\tr}{tr}
\DeclareMathOperator{\ccc}{ccc}

\nc{\fA}{{\mathfrak{A}}}
\nc{\fB}{{\mathfrak{B}}}
\nc{\fC}{{\mathfrak{C}}}
\nc{\fD}{{\mathfrak{D}}}
\nc{\fE}{{\mathfrak{E}}}
\nc{\fF}{{\mathfrak{F}}}

\nc{\p}{\partial}
\nc{\ph}{\hat{\partial}}

\nc{\war}{{\color{red} CHECK}}
\nc{\rtext}[1]{{\color{red}{#1}}}

\author[S. Torelli]{Sara Torelli}
\address[S. Torelli]{Dipartimento di Matematica Giuseppe Peano, Università di Torino, Via Verdi 8, 10125 Torino, Italia}
\email{sara.torelli7@gmail.com}

\date{July 30, 2023}


\thanks{
This work was supported by the Alexander von Humboldt Foundation. 
}

\title{Correspondences acting on constant cycle curves on K3 surfaces}

\begin{document}

\maketitle
\begin{abstract} Constant cycle curves on a K3 surface $X$ over $\bC$ are curves whose points all define the same class in the Chow group. In this paper we study correspondences $Z \subseteq X\times X$ over $\bC$ acting on the group $\ccc(X)$ of cycles generated by irreducible constant cycle curves. We construct for any $n\geq 2$ and any very ample line bundle $L$ a locus $Z_n(L)\subseteq X\times X$ which is expected to have dimension $2$ and which yields a correspondence that acts on $\ccc(X)$, when it has dimension $2$. We provide examples of $Z_n(L)$ for low $n$ and exhibit one correspondence different from $Z_n(L)$ acting on $\ccc(X)$.
\end{abstract}

\section{Introduction.}
Constant cycle curves have been considered in \cite{Voi15} and formally introduced in \cite{Huy14} with the motivation of getting a more thorough picture of Chow groups and rational curves on K3 surfaces. On a K3 surface $X$ defined over $\bC$, they are curves whose points all define the same class in the Chow group of zero cycles of $X$. Rational curves are important examples of constant cycle curves. On the other hand, examples of non-rational constant cycle curves are constructed in \cite{Huy14}. Overall, constant cycle curves behave similarly to rational curves. They are rigid, as the existence of a family dominating $X$ would force $CH_0(X)$ to be isomorphic to $\bZ$, contradicting Mumford's theorem (see \cite{Mum68}). They are finite in linear systems, for bounded order, as proven in \cite{Huy14}. Their union is even dense with respect to the classical topology (see \cite{Huy14} or directly \cite{Voi15}). Yet the set of constant cycle curves on K3 surfaces is not well understood. 
%In the paper we define the {\em group of constant cycle cycles} as the free abelian group subgroup $\ccc(X)\subseteq Z_1(X)$ of the group of $1$-cycles generated by irreducible constant cycle curves. 

In this paper, we study correspondences $Z \subseteq X\times X$ over $\bC$ preserving the set of constant cycle curves. More precisely, we introduce the group $\ccc(X)\subseteq Z_1(X)$ of cycles generated by irreducible constant cycle curves (see Definition \ref{def:setccc}). Then we study correspondences $Z \subseteq X\times X$ {\em acting on $\ccc (X)$}, i.e. inducing a homomorphism $Z_*: \ccc(X)\to \ccc(X)$ (see Definition \ref{def:corrccc}).  

%A priori, it is not clear that correspondences of this kind exist.
Notice that the notion of constant cycle curve is not well defined at the level of Chow groups, since constant cycle curves are rigid. This means that not all curves that are rationally equivalent to a constant cycle curve are constant cycle curves. It is therefore essential to work at the level of cycles. 
%Namely, we look for correspondences whose homomorphism $Z_*$ defined on $CH_1(X)$ lifts to the level of cycle. %Only for them we can ask whether they preserve the group under consideration. 

Natural examples to look at are $2$-cycles $Z\subseteq X \times X$ such that for any point $(p,q)\in Z$, $p$ is rationally equivalent to $q$ in $X$. These correspondences have been studied before (see \cite{Voi04}, \cite{M04}). %Nonetheless, in this generality these cycles are very hard to study explicitly in order to deduce further properties on constant cycle curves.

In the paper, we focus on examples of cycles defined by curves of genus $g\geq 3$. We construct for any $n\geq 2$ and any very ample line bundle $L$ a cycle $Z_n(L)\subseteq X\times X$ that is expected to have dimension $2$ and induces a correspondence that acts on $\ccc(X)$, when of dimension $2$. We then provide examples for low $n$ having dimension $2$. 
%We also produce one correspondence acting on $\ccc(X)$ which is different from $Z_n(L)$.
%\begin{conjecture}\label{conj:ccc} The set of correspondences on a K3 surface $X$ acting on the group of constant cycle cycles $\ccc(X)$ is at least countable.
%\end{conjecture}
%In the paper we restrict our attention to smooth projective K3 surfaces $X$ over $\bC$. We construct correspondences of expected dimension $2$ as in the conjecture for any integer $n\geq 2$ and we prove  that for low $n$ they have the expected dimension on a concrete K3 surface. We are convinced that these correspondences provide good candidates supporting the conjecture. 

%More precisely we can state our first main result. 
\begin{theorem}\label{thm:main1} Let $X$ be a smooth projective K3 surface over $\bC$ and $L$ a very ample line bundle. Then for any integer $n\geq 2$ the locus $Z_n(L)\subseteq X\times X$ defined as the closure of
$$ Z_n(L)^0= \{(p,q)\in X\times X\, |\,  \exists \, C\in |L| \mbox{ smooth s.t. }p,q\in C, p\neq q, \,n[p-q]=0\in JC\}$$
is a locus of expected dimension $2$. Moreover, $Z_n(L)$ acts on $\ccc(X)$ when it has dimension $2$ (see Definition \ref{def:corrccc}).
\end{theorem}
%In this paper we restrict our attention to study the problem over the complex numbers, namely we work over $\bC$, and on smooth projective K3 surfaces (defined over $\bC$). The starting point is given by an example of non rational constant cycle curve given by \cite{Huy14}, inspired to \cite{Voi}, and defined on an elliptic K3 surface $X$ with a $0$ section as the locus $C_n\subseteq X$ of $n$-torsion points on the fibres of the elliptic fibration. This example is a very natural one to understand how constant cycle curves can occur to be not rational, since it is somehow constructed starting from a rational curve. The $0$ section of $X$ is in fact a rational curve $R\subseteq X$. Now any point $x_n\in C_n$ is a $n$-torsion point of a fibre $E_t$ for some $t\in R$, namely $n(x_n-t)$ is rationally equivalent to $0$ on $E_t$, and so also on $X$. Since the Chow group of K3 surfaces is torsion free, one concludes that $x_n$ is rationally equivalent to $t$ and since all points of a rational curve are rationally equivalent, all points of $C_n$ are as well, proving that $C_n$ is a constant cycle curve.
The construction of the correspondences $Z_n(L)$ generalizes ideas of \cite{Huy14} and \cite{Voi15}, that the locus of $n$-torsion points of the fibers of an elliptic fibration $X\to \bP^1$ defines a constant cycle curve on $X$. 

%The loci that we construct crucially differ from those from \cite{Huy14} in that they have expected dimension 2 rather than dimension $1$. 
Different to \cite{Huy14}, the loci we construct are expected to be of dimension $2$ when not empty rather than $1$. Notice that for any point $(p,q)\in Z_n(L)$, $p$ is rationally equivalent to $q$ and so all points $(p,q)\in Z_n(L)$ lie on a fibre of the difference map $X\times X \to CH_0(X)$. Mumford proved that the dimension of this fibre is at most $2$. Nonetheless, whether the locus $Z_n(L)$ is not empty or of dimension $1$ or $2$ is not yet determined.

%NEW The expected dimension of $Z_n(L)$ is computed in the moduli space $\cM_{g,2}$ of genus $g$ curves with two marked points, by using the rational map $|L| \dashrightarrow \cM_g$, which is generically finite for very ample line bundles by \cite{Bak22}. The locus $Z_n(L)$ is determined in $\cM_{g,2}$ as the intersection with a certain Hurwitz locus in $\cM_{g,2}$. 

The proof of Theorem \ref{thm:main1} is contained in Section \ref{sec:Zn}, and follows from a series of more precise statements (see Theorems \ref{thm:dimZn}, \ref{thm:expZn}, \ref{thm:Znccc}).

Determining whether $Z_n(L)$ is non-empty is not trivial, as the difference map $C\times C\to JC$ is not in general surjective for $g\geq 3$ (as it was for elliptic curves). Therefore, the existence of points $p,q\in X$ defining $n$-torsion points $p-q$ of $JC$ is no longer guaranteed, for $g\geq 3$. Notice that when $m$ divides $n$, $Z_m(L)$ defines a component of $Z_n(L)$, as $m$-torsion points are also $n$-torsion points. Therefore, the problem of non-emptiness for $n$ is in this case reduced to that for $m$. However, the investigation of a component in $Z_n(L)$ cut out by $n$-torsion points that are not $m$-torsion points remains an interesting problem. 
The problem can be interpreted as a more refined problem of Brill--Noether theory.
%NEW, or rather of Hurwitz spaces.
%NEW As mentioned above, the locus $Z_n(L)$ is determined in $\cM_{g,2}$ as the intersection with a certain Hurwitz locus in $\cM_{g,2}$. 
This because a point $(p,q)\in Z_n(L)$ is given by $p,q\in C$, for $C\in |L|$, with the property that $n[p-q]=0\in JC$. The latter condition defines a degree $n$ map $C\to \bP^1$ totally ramified at $p$ and $q$, which is in particular a $g^1_n$. Recall that a Brill--Noether general curve admits a $g^1_n$ if and only if the Brill--Noether number is non-negative, which happens precisely for $g\leq 2n-2$ (see \cite{ACGH85}). By \cite{Laz86}, every smooth curve in the linear system $|L|$ is Brill--Noether general when $\Pic X$ is generated by $L$. Therefore, for a very general K3 surface $X$, $n$ must satisfy $g\leq 2n-2$ for $Z_n(L)$ to be non-empty. The important observation is that these smooth curves do not necessarily admit a $g^1_n$ with  the special ramification profile induced by points of $Z_n(L)$.

%The proof of this by \cite{Laz86} relies in fact on K3 surfaces and states something even more precise. Namely, that on any linear system $|L|$ containing only reduced irreducible curves, the general element satisfies the Petri condition. This implies that all these curves admit a $g^1_n$ for $g\leq 2n-2$, giving a ramified morphism $C\to \bP^1$. 

\begin{conjecture}\label{conj:ccc} Let $(X,L)$ be a very general polarized K3 surface of genus $g\geq 3$. Then $Z_n(L)$ is not empty for any $g\leq 2n-2$. 
\end{conjecture}

The case of negative Brill--Noether number, i.e. $g\geq 2n-2$, is also interesting on K3 surfaces of Picard rank at least two. Indeed, curves of every gonality in certain K3 surfaces of Picard rank $2$ have been constructed in \cite{Knu03}. Again the examples do not satisfy the ramification profile induced by $Z_n(L)$ in general, but they might be worth further investigation. 

In the paper, we confirm the above conjecture for $n=3,4$ in the following cases. 
\begin{theorem}\label{thm:main2} Let $X\subseteq \bP^3$ be a smooth quartic not containing a line and $L$ a hyperplane section. Then $Z_3(L)$ is non-empty of dimension $2$.
\end{theorem}
\begin{theorem}\label{thm:main3} Let $X\subseteq \bP^4$ be a general smooth complete intersection of a cubic and a quadric and $L$ a hyperplane section. Then $Z_4(L)$ is non-empty of dimension $2$.  
\end{theorem}
The proofs of these theorems are contained in Sections \ref{sec:J} and \ref{sec:T} and rely on the study of two further correspondences $J$ and $T$, related to the geometry of these concrete examples. 
Furthermore, since in genus $g=2$ the line bundle $L$ of Theorem \ref{thm:main1} cannot be very ample and therefore the assumptions of the theorem are not satisfied, we devote Section \ref{sec:Z2} to carry out by hand the example for $n=2$ and $L$ the pullback of a line via the degree $2$ cover from the K3 surface to $\bP^2$.

It is worth mentioning that the correspondences $Z_n(L)$ satisfy the so-called easy direction of the generalized Bloch conjecture (see Subsection \ref{subsec:BB} or directly \cite{Voi2}). Namely, they act by multiplication on the Chow group of $0$-cycles, thus they act by multiplication on the transcendental cohomology.
\begin{theorem} In the setting of Theorem \ref{thm:main1}, $Z_n(L)_*\colon CH_0(X)\to CH_0(X)$ acts by multiplication, if $Z_n(L)$ has dimension $2$. In particular, $Z_n(L)^*$ acts by multiplication on $H^2_{\tr}(X,\bQ)$.
\end{theorem}

We would like to spend a couple of words on the correspondences $J$ and $T$ involved in the proofs of Theorems \ref{thm:main2},\ref{thm:main3}, because they are quite interesting on their own.

%The correspondence $J$ is fully studied in section \ref{sec:J}. This correspondence is particularly interesting since it satisfies more or less all the properties that $Z_3$ does. For instance, we can prove 
In fact, we can use $J$ to prove the existence of at least one correspondence different from $Z_n(L)$ of Theorem \ref{thm:main1} acting on the group $\ccc(X)$.
More importantly, $J\subseteq X\times X$ does not satisfy the property that for any $(p,q)\in J$, $p$ is rationally equivalent to $q$, providing an example of a correspondence acting on constant cycle curves, which does not arise from the natural ones explained in the previous part of the introduction.

\begin{theorem}\label{thm:main5} Let $X$ be a smooth quartic surface not containing a line. Define $J\subseteq X\times X$ as the closure of the locus 
\begin{equation}
    J^0= \{(p,q)\in X\times X \,|\, p\neq q,  X\cdot l_{p,q}=3p+q \in Z_0(X)\},
\end{equation}
where $l_{p,q}$ denotes the line between $p,q$. Then $J$ is a $2$-cycle which yields a correspondence with generically finite projections and acting on the group $\ccc(X)$. 
\end{theorem}
As an application, this correspondence allows in some special cases to construct non-rational constant cycle curves starting from smooth conics, which are rational curves.
%\begin{theorem}\label{thm:main6}
%Assume that a smooth quartic surface $X$ contains a smooth conic $Q$ s.t.
%\begin{enumerate}\item $Q$ is transversal to the branch locus of the $1$st projection of $J$;
%\item the preimage of $Q$ via the $1$st projection is birational to the image via the $2$nd projection.
%\end{enumerate} 
%Then the homomorphism $J_*:\ccc (X)\to \ccc(X)$ maps $Q$ to a non rational constant cycle curve.
%\end{theorem}
The proof of the last theorem is contained in Section \ref{sec:J}. We will actually see that $Z_n(L)$ and $J$ share some interesting properties, concerning the study of constant cycle curves. 

This fact is no longer true for the correspondence $T$ studied in Section \ref{sec:T}, whose definition is a natural generalization of $J$ to the case of a K3 surface realized as a complete intersection in $\bP^4$. Namely, we look for triples of points $(p,q,r)$ on the surface whose intersection with the plane generated by these points is $4p+q+r$. Already as a locus $T\subseteq X\times X \times X$ , $T$ is not fitting our previous setting.

%The correspondence $T\subseteq X\times X\times X$ is studied in section \ref{sec:T}. We can prove in this case less similarities with the correspondences $Z'_n$ and the reasons will be better clarified along the section. More or less they are consequences of the fact that we can't define $T$ directly on the product of two copies of $X$ but for intersection reasons we must consider $3$ copies instead.

To conclude, notice that there is one further explicit presentation of a K3 surface in $\bP^5$ as a complete intersection of $3$ quadrics. One would wonder whether we can generalize the constructions of $J$ and $T$ to this case, by looking at $3$-planes intersecting the K3 surface in a point of multiplicity $5$. This however would need a new argument.
%Then use it to deduce non emptiness of $Z_5(H)$ with respect to the hyperplane section $H$. This would be very interesting to figure out but we can't do it by using the same techniques used for $J$ and $T$. The main reason is that points intersecting with multiplicity higher than $2$ must lie in the tangent space of the surface. In both cases of lines for $J$ and planes for $T$ we could reduce the problem to look at linear subspaces of the tangent space. This is no longer true once one wants to consider $3$-planes and so one needs to work out a different argument.

\vspace{2pt}
{\bf Acknowledgments.} I am grateful to Stefan Schreieder, who introduced me to the topic and stimulated the preparation of the paper with many insightful conversations. I would also like to thank Gian Pietro Pirola, Claire Voisin and Daniel Huybrechts, for their comments on the paper and conversations we shared on the topic, and the anonymous referee for their very helpful comments. \\

\section{Preliminaries.}\label{sec:prel}
\subsection{Correspondences on smooth projective surfaces.}\label{sub:Notations}
Let $X$ be a smooth projective surface over $\bC$, $Z_i(X)$ the group of cycles of dimension $i$ and $CH_i(X)$ the Chow group. Let $Z\subseteq X\times X$ be an irreducible subvariety and $p_{i,X}:X\times X\to X $ the projection on the $i$-th factor. We denote by $p_{i,Z}:Z\to X$ the restrictions of $p_{i,X}$ to $Z$ and call them projections from $Z$. 

For any $2$-cycle $Z\in Z_2(X\times X)$, Fulton's refined intersection (see \cite[Chapter 8]{Fu}) gives a homomorphism 
\begin{equation}Z_*:CH_i(X)\to CH_i(X), \quad Z_*[D]=p_{2,X*}(Z\cdot p_{1,X}^*([D])),
\end{equation}
which is compatible with the schematic inverse image in the following sense. Let $D\in Z_1(X)$ be reduced irreducible, then 
\begin{equation} Z_*[D]\in CH_1(p_{2,X}(|Z|\cap p_{1,X}^{-1}(|D|))).
\end{equation}
More precisely, on any $D\in Z_1(X)$ reduced irreducible, the pullback $p_{1,X}^*(D)$ is well defined without any assumption of flatness and gives a class $p_{1,X}^*(D)\in CH_*(p_{1,X}^{-1}(|D|))$. The intersection product $Z\cdot p_1^*(D)$ given by the Fulton's refined intersection is obtained by taking the diagonal $\delta_{X\times X}$ and applying the refined Gysin homomorphism $\delta_{X\times X}^{!}$ to the class $Z\times p_{1,X}^*(D)$ in $CH_*(X\times X\times X\times X)$. We obtain a cycle, 
\begin{equation}\label{mor:Z_*fult}
Z\cdot p_{1,X}^*(D)= \delta_{X\times X}^{!}(Z\times p_{,X1}^*(D))\in CH_*(|Z|\cap p_{1,X}^{-1}(|D|))
\end{equation}
with support in $|Z|\cap p_{1,X}^{-1}(|D|)$.
By push forward we have a cycle $p_{2,X*}(Z\cdot p_{1,X}^*(D))$ with support in $p_{2,X}(|Z|\cap p_{1,X}^{-1}(|D|))$ as wanted. The definition is extended by linearity.

Overall, this refinement gives us a class with a precise support $p_{2,X}(|Z|\cap p_{1,X}^{-1}(|D|))$, for a given precise support $|D|$. Furthermore, it allows to define pullbacks of non-flat projections. We use this refinement to lift $Z_*$ at the level of cycles under certain assumptions.

\begin{lemma}\label{lem:Z_*fult} Let $X$ be a smooth projective surface and $Z\in Z_2(X\times X)$ be an irrecucible reduced subvariety. Assume that one of the following holds:
\begin{enumerate}
\item $p_{2,X}(Z)\subseteq X$ has dimension $\leq 1$;
\item $p_{1,X}$ is generically finite.
\end{enumerate}
Then there exists a homomorphism $Z_*:Z_1(X)\to Z_1(X)$ making the diagram 
\begin{equation}\label{dia:Z*}
\xymatrix{
	Z_*: Z_1(X)\ar[d] \ar[r]& Z_1(X) \ar[d] \\
	Z_*: CH_1(X)\ar[r] & CH_1(X)
}
\end{equation}
commutative.
\end{lemma}
\begin{proof}
All we have to check is that to any irreducible curve of $C\in Z_1(X)$, we can associate a cycle $Z_*C\in Z_1(X)$ such that $[Z_*C]=Z_*[C]\in CH_1(X)$. Namely, a cycle $Z_*C$ making Diagram  \eqref{dia:Z*} commutative. Then the statement follows by linearity on the intersection product.

Let $C\in Z_1(X)$ be an irreducible reduced curve. The refined intersection of Fulton gives on Chow groups a class $Z_*[C]\in CH_1(p_{2,X}(|Z|\cap p_{1,X}^{-1}(|D|)))$.

Assume that any irreducible component of the support $p_{2,X}(|Z|\cap p_{1,X}^{-1}(|D|))$ has dimension $\leq 1$. Then we can define a cycle $Z_*C$ as it follows. If the support is either empty or zero dimensional, we set $Z_*C=0$, which is compatible with the definition of Chow groups. Otherwise the support is a curve, after eventually removing isolated points, which do not contribute to define $Z_*[C].$ Since the Chow group of a curve is generated by its irreducible components, we can recover from its generators a cycle $C'$ in $Z_1(X)$. Setting $Z_*C=C'$ we obtain a cycle, which by definition makes the diagram commutative. 

Therefore, we just have to check that the support $p_{2,X}(|Z|\cap p_{1,X}^{-1}(|D|))$ has dimension $\leq 1$. This is obvious if $p_{2,X}(Z)\subseteq X$ has dimension $\leq 1$. Otherwise, $p_{1,X}$ is generically finite by assumption. Thus  $|Z|\cap p_{1,X}^{-1}(|D|)$ is already a curve and we conclude.
\end{proof}

\subsection{Constant cycle curves and correspondences on K3 surfaces.}\label{subsec:ccc}
Constant cycle curves have been formally introduced in \cite{Huy14} and they can be presented using a few different equivalent definitions. In the paper, we work over $\bC$ and we will always use the following definition.
\begin{definition}\label{def:ccc} Let $X$ be a smooth projective K3 surface. A curve $C\subseteq X$ is a constant cycle curve if all points $p\in C$ define the same class $[p]\in CH_0(X).$ 
\end{definition}
It is easy to see (see \cite{Huy14}) that on any K3 surface $X$ any point of a constant cycle curve defines the distinguished class of Beauville--Voisin $c_X$ (see \cite{BV04}) and so the class of a point of a rational curve. 
%It is shown in \cite{Huy14} that constant cycle curves are rigid, that rational curves are constant cycle curves but the converse is not true. 
%For instance, on elliptic K3 surfaces non rational constant cycle curves are defined by using $n$ torsion points of elliptic curves. Overall it is not clear yet how to characterize constant cycle curves. 

In the paper, we want to study the constant cycle curves under the action of suitable correspondences. 
\begin{definition}\label{def:setccc} We define the {\em group $ccc(X)$ generated by constant cycle curves} as the free abelian subgroup $$\ccc(X)\subseteq Z_1(X)$$ of the group of $1$-cycles generated by irreducible constant cycle curves. We say that $C\in \ccc(X)$ is a constant cycle cycle.
\end{definition}
%\begin{problem}\label{prob:ccc} Understand properties on the group $\ccc(X)$ generated by constant cycle curves.
%\end{problem}
%We want to address the problem by looking at suitable correspondences, namely correspondences preserving the property that all points of any reduced component of a cycle give the same class.

\begin{definition}\label{def:corrccc} We say that a correspondence $Z\in Z_2(X\times X)$ {\em acts on constant cycle curves (or on $\ccc(X)$)} if the natural homomorphism $Z_*:CH_1(X)\to CH_1(X)$ induces a homomorphism $$Z_*\,:\,\ccc(X)\to \ccc(X).$$  
In other words, if for any irreducible constant cycle curve $C$, $Z_*C$ is a well defined $1$-cycle that can be written as a linear combination of irreducible constant cycle curves.
\end{definition}
In particular, we want to understand how many of these correspondences we can find on a K3 surface and whether we can use them to construct non-rational constant cycle curves acting on rational curves. As said in Conjecture \ref{conj:ccc} in the introduction, we expect to find many of them.
\subsection{Generalized Bloch conjecture.}\label{subsec:BB}

The Generalized Bloch conjecture generalizes the classical Bloch conjecture for surfaces (see e.g. \cite[Conjecture 11.2]{Voi2}) to the case of $2$-dimensional correspondences in the product of two possibly different smooth projective surfaces. 
%It is also a weaker version of the Bloch-Beilinson conjecture (see e.g. \cite[Conjecture 11.21]{Voi2}). 

In the paper, we restrict our attention to the case of the product $X\times X$ of two copies of a smooth projective surface $X$. Let $CH_0(C)_{\homo}\subseteq CH_0(X)$ denote the subgroup of $0$-cycles homologically equivalent to $0$, modulo rational equivalence, and let $\alb_X: CH_0(X)_{\homo}\to \Alb(X)$ be the Albanese map restricted to $CH_0(C)_{\homo}$. Define a filtration on $CH_0(X)$ as
\begin{equation}\label{eq:Chowfilt}
F^0CH_0(X)=CH_0(X), \quad F^1CH_0(X)= CH_0(X)_{\homo}, \quad F^2CH_0(X)=\ker \alb_X \subseteq F^1CH_0(X).
\end{equation}
Recall that since $X$ is projective, there are well defined homomorphisms 
\begin{equation}\label{mor:pullclass} Z^*: CH_0(X)\to CH_0(X), \quad Z^*([Y])= [{\pi_{1,X}}_*(Z\cdot \pi_{2,X}^*(Y))]
\end{equation} 
and $[Z]^*: H^2(X,\bC)\to H^2(X,\bC)$.
%Composing both sides with the class map $\cl: CH_0(X)\to H^4(X, \bZ)\subseteq H^2(X,\bC)$, we obtain a commutative diagram 
%\begin{equation}\label{def:classhom}
%\xymatrix{
%    Z^*\,: CH_0(X) \ar[r]  \ar[d]_{\cl}&
%     CH_0(X)\ar[d]_{\cl}\\$\overline{\M}_g$
%    [Z]^*\,: H^4(X,\bZ) \ar[r] & 
%    H^4(X,\bZ) 
%}
%\end{equation}
%defining the homomorphism $[Z]^*:H^2(X,\bC)\to H^2(X,\bC)$.

In this setting the Generalized Bloch conjecture (see \cite[conjecture 11.19]{Voi2}) is rephrased as it follows.

\begin{conj}[Generalized Bloch conjecture] For any $2$-cycle $Z \subseteq X\times X$ such that $[Z]^*:H^2(X, \bC)\to H^2(X, \bC)$  vanishes on $H^0(X, \Omega^2_X)$, the morphism $Z_*:CH_0(X)\to CH_0(X)$  vanishes on $F^2CH_0(X)$.
\end{conj}
The reader can find the definition of $[Z]^*$ in \cite[Chapter 11]{Voi2} and of $Z_*$ in Subsection \ref{sub:Notations}.
%and more precisely looking at diagram \ref{def:classhom} and formula \ref{lem:Z_*fult}.

The inverse of the Generalized Bloch--Beilinson conjecture is easier and already established (\cite[proposition 11.18]{Voi2}). We recall here the statement in view also of \cite[remark 11.20]{Voi2}. 
\begin{lemma}\label{lem:Voi11.18} Let $Z$ be a $2$ cycle on $X$ and assume that $Z_*$ vanishes on $F^2CH_0(X)$. Then  $[Z]^*: H^2(X,\bC)\to H^2(X,\bC)$ vanishes on $H^0(X,\Omega^2_X)\subseteq H^2(X,\bC)$ and furthermore $Z^*$ vanishes on $H^2(X, \bQ)_{\tr}$, namely on the smallest Hodge substructure containing $H^0(X,\Omega^2_X)$.
\end{lemma}
As an application, consider a correspondence $Z$ and an integer $d$. Take the correspondence $Z''=Z - d\Delta_X$, where $\Delta_X \subseteq X\times X $ denotes the diagonal. Then $Z''$ satisfies the assumptions of Lemma \ref{lem:Voi11.18} if and only if $Z_*$ behaves as a multiple of the diagonal. We will see examples in the next sections.

\section{Correspondences from $n$-torsion points of linear systems on $K3$ surfaces.}\label{sec:Zn} 
In this section we introduce for any integer $n\geq 2$ a locus $Z_n(L)\subseteq X\times X$ parametrizing $n-$torsion points of curves $C$ in the linear system  $|L|$ on a $K3$ surface $X$. We prove that the expected dimension of such a locus is $2$. When the locus $Z_n(L)$ has dimension $2$, we show that it acts on the group $\ccc(X)$ generated by constant cycle curves (Definition \ref{def:corrccc}) and that it satisfies the so-called easy direction of the generalized Bloch conjecture (Subsection \ref{subsec:BB}). 

\subsection{The loci $Z'_n(L)$ and $Z_n(L)$.}\label{subsec:Z_n} Let $X$ be a $K3$ surface over $\bC$ and $L$ be a very ample line bundle. In particular, $L^2>2$ and any smooth irreducible curve $C\in |L|$ has genus $g= \frac{L^2}{2} +1>2$. 

Let $\cC\to |L|$ be the universal family over the linear system $|L|$. The fiber product $$\cC\times_{|L|}\cC\to |L|$$ parametrizes triples $(C,p,q)$ such that $C\in |L|$, $p,q\in C$. We denote by $\cC^0\to |L|^0$ and $\cC^0\times_{|L|^0}\cC^0\to |L|^0$ the restrictions to the locus of smooth irreducible curves $|L|^0$ in $|L|$. 
%Since $X$ is a K3 surface, $|C|$ has dimension $g$ and $\cC\times_{|C|}\cC$ has dimension $g+2$.
\begin{definition}\label{def:Zn}
We define the {\it locus $Z'_n(L)$ of {\em $n-$torsion points of a linear system $|L|$}} as closure in $\cC\times_{|L|}\cC$ of the locus
\begin{equation}
    Z'^0_n(L):= \left\{ (C,p,q)\in \cC^0\times_{|L|^0}\cC^0\ \ |\ \ p\neq q, \,  n[p-q]=0\in JC\right\}\subseteq \cC\times_{|L|}\cC.
\end{equation}
%In words, $Z_n$ is the locus of $n-$torsion points for some curve $C$ in the linear system $|C|$.
\end{definition}

These loci have a natural interpretation in terms of coverings of curves.   
Let $\cM_{g,2}$ be the moduli space of genus $g$ smooth projective curves with 2 marked points and let $\overline{\cM_{g,2}}$ be the Deligne Mumford compactification, i.e. the moduli space of stable curves with $2$ marked points. We define the locus $\cH_n\subseteq\overline{\cM_{g,2}}$ as the closure of the locus  
\begin{equation}\label{eq:Hn}
    \cH^0_n:= \left\{ (C,p,q)\in \cM_{g,2}\, | \,\exists\, \pi: C\to \bP^1 \mbox{ cover}, \, \deg \pi=n, R(\pi)\supseteq (n-1)(p-q) \right\},
\end{equation}
where $R(\pi)$ denotes the ramification divisor. In words, $\cH^0_n$ is the locus of two pointed smooth curves admitting a degree $n$ cover totally ramified at the two marked points. 

By definition, a point $(C,p,q)\in Z'^0_n(L)$ satisfies the condition $n[p-q]=0\in JC$. This condition defines a cover $\pi: C\to \bP^1$ of degree $n$, ramified at $p$ and $q$ with total ramification. Thus, a point $(C,p,q) \in\cH^0_n.$ 

Consider now the natural modular rational map   
\begin{equation}\label{eq:mapMg2}
    \phi: \cC\times_{|C|}\cC\dashrightarrow \overline{\cM_{g,2}},
\end{equation} 
defined by sending a triple $(C,p,q)$ to the stable model of $C$ with two marked points $p,q$.  We use it to study the image of the loci $Z'_n(L)$ in $\overline{\cM_{g,2}}$ and more precisely in the loci $\cH_n$.  

\begin{theorem}\label{thm:dimZn}
Let $X$ be a K3 surface and $L$ a very ample line bundle.  
%Let $\cH_n$ be as in \ref{eq:Hn}. Then 
%\begin{equation}\label{eq:ZnHn}
%    \phi(Z'_n(L))=\cH_n\cap \phi(\cC\times_{|L|}\cC).
%\end{equation}
%In particular, 
Then the expected dimension of $Z'_n(L)$ is $2.$
\end{theorem}
\begin{proof}  
%$\phi(Z'_n(L))\subseteq \overline{\phi(Z'^0_n(L))}\subseteq \cH_n$ and we obtain formula \ref{eq:Hn}.

Recall that points of $\cC\times_{|L|}\cC$ are just triples $(C,p,q)$, where $C\in |L|$ and $p,q\in C$. The rational map $\phi:\cC\times_{|L|}\cC\dashrightarrow \overline{\cM_{g,2}}$ sends a point $(C,p,q)\in \cC\times_{|L|}\cC$ to $(C,p,q)\in \overline{\cM_{g,2}}$. We want to prove that this rational map is generically finite. Suppose that $\phi(C,p,q)=\phi(C',p',q')$. Then $C$ and $C'$ are isomorphic, and there exists an automorphism $f$ of $C$ such that $f(p)=p'$ and $f(q)=q'$. Since the group of automorphisms of a curve of genus $g\geq2$ is finite, we only have finitely many choices of $p',q'$ for given $p,q$. Therefore, to prove that this map is generically finite, it suffices to prove that generically we have only finitely many choices of $C\in |L|$ in the fiber of $C\in \overline{\cM_g}$. In other words, it suffices to prove that the rational map $|L|\dashrightarrow \cM_{g}$ is generically finite. This fact follows by \cite{Bak22}, since we work under the assumption that $L$ is very ample. 

As $\cC\times_{|L|}\cC\dashrightarrow \overline{\cM_{g,2}}$ is generically finite, we can compute the expected dimension of $Z_n(L)$ in $\overline{\cM}_{g,2}.$
%By \cite[corollary 7]{Bak22}, the rational map $|L|\dashrightarrow \cM_{g}$ is generically finite. This implies that the rational map $\cC\times_{|L|}\cC\dashrightarrow \overline{\cM_{g,2}}$ is also generically finite. Thus we can compute the expected dimension of $Z_n(L)$ in $\overline{\cM}_{g,2}.$

As observed above, a point $(C,p,q)\in Z'^0_n(L)$ defines a point in $\cH^0_n,$ as defined in \eqref{eq:Hn}. Namely, the condition $n[p-q]=0\in JC$ defines a cover $\pi: C\to \bP^1$ of degree $n$, totally ramified at $p$ and $q$. Thus, the expected dimension of $Z'_n(L)$ is given by 
\begin{equation}
    \dim \cH_n + \dim ( \cC\times_{|L|}\cC) - \dim \overline{\cM_{g,2}},
\end{equation}
since all these varieties have only quotient singularities. Notice that $\cH_n$ has dimension $2g-1$. Indeed, the ramification divisor $R(\pi)$ depends in general on $r=2g-2+2n$ parameters. Then we must subtract the contribution $2(n-1)$ given by the two (fixed) marked points $p,q$, and take into account the action of the space of automorphisms of $\bP^1$ that fix the two marked points $p,q$, which is of dimension $1$. 

The dimension of $\cC\times_{|L|}\cC$ is $g+2$. Indeed, $X$ is a K3 surface, $|L|$ has dimension $g$. Since the dimension of $\overline{\cM_{g,2}}$ is $3g-1$, putting all together we conclude that the expected dimension of $Z'_n(L)$ is $2$. 
\end{proof}
%\begin{remark} A very important point of the proof is that the linear system $|L|$ defines a generically finite rational map $\cC\dashrightarrow \cM_g$. This is a very subtle problem that has been intensively investigated on K3 surfaces. The recent result contained in \cite{Bak22} answers the problem for any base point free ample line bundles on any K3 surface. Thus, the assumption of maximal variation we put in the statement is now not necessary, but it is worth to be mentioned, in order to stress its important on the statement. Many other partial results have been given in this direction, first of all in the result of \cite{VV16}, improved by \cite{Bak22}. Then also by using different approaches as in  \cite{Muk92, DH21, ABS14, ABS17, CDS20, CD20}.
%\end{remark}
\begin{remark} As a bibliographic information, the behavior of $|L| \dashrightarrow \cM_g$ for $L$ very ample was already known before \cite{Bak22} in many cases, established for instance by \cite{Muk92, DH21,ABS14,ABS17,CDS20,CD20}. 
\end{remark}

 We now define a cycle  $Z_n(L)\subseteq X\times X$ starting from $Z'_n(L)$. Consider the forgetful map $\psi: \cC\times_{|L|}\cC\to X\times X$, sending $(C,p,q)$ to $(p,q)$.  %This induces a map on cycles \begin{equation}\psi_\ast:Z_*(\cC\times_{|L|}\cC)\to Z_*(X\times X).\end{equation}
\begin{definition}\label{def:psiZ_n}
We define $Z_n(L)\subseteq X\times X$ as $Z_n(L):= \psi (Z'_n(L))$.
Moreover, we denote by $p_{i,Z_n(L)}\colon Z_n(L)\to X$ the projection over the $i$-th factor, for $i=1,2$.
\end{definition}

Explicitly, by definition  $Z_n(L)$ is realized as the closure of the locus 
\begin{equation}\label{eq:Z0n}
Z^0_n(L)= \{(p,q)\in X\times X\, |\,  \exists \, C\in |L| \mbox{ smooth s.t. }p,q\in C, p\neq q, \, \,n[p-q]=0\in JC\},
\end{equation}
as in the statement of Theorem \ref{thm:main1}.

Notice that a priori $Z_n(L)$ and $Z'_n(L)$ can have different dimension if the map $\psi$ is not generically finite over $Z_n'(L)$. Let $|L_{p,q}|$ the linear system of $C\in |L|$ passing through $p,q$. For general $p,q$, the rational map $|L_{p,q}|\dashrightarrow \mathcal{M}_g$ is genrically finite. In the next theorem we will see how this leads to conclude that $Z_n(L)$  and $Z'_n(L)$ share the same expected dimension.

%For instance, this happens if for a general $(p,q)\in Z_n(L)$, the linear system of curves $|C'-p-q|$ for $C'\in |L|$ admits a positive dimensional subvariety $Y$ of curves $C_t$ with the property that $n(p-q)=0\in JC_t$. Nonetheless, it is reasonable it expect this to happen rarely. One reason is that as already remarked in a more general setting (see remark \ref{rem:finproj}), this can't happen when $Z'_n(L)$ has finite projections.
\begin{theorem}\label{thm:expZn} The expected dimensions of $Z_n'(L)$ and $Z_n(L)$ coincide. 
\end{theorem}
\begin{proof} Suppose that $Z_n(L)$ has expected dimension smaller than $Z_n'(L)$. This means that the general fiber of the restriction $\psi: Z'_n(L)\to Z_n(L)$ to $Z'_n(L)$ of $\psi:\cC\times_{|L|}\cC\to X\times X$ has positive dimension. Namely for a general $(p,q)\in Z_n(L)$, the linear system $|C'-p-q|$ for $C'\in |L|$ and $p,q\in C'$ contains a positive dimensional subvariety $Y_{pq}$ of curves with the property that $n[p-q]=0\in JC_t$ for any $C_t\in Y\subseteq |C'-p-q|$. 

We now prove that $Y$ has expected dimension $0$. Recall that $|L_{pq}|$ is the linear system  $|C'-p-q|$, let $\cC\to |L_{pq}|$ be the universal family over $|L_{pq}|$ and consider the fiber product $\cC\times_{|L_{pq}|}\cC\to |L_{pq}|$. We have a rational map 
$$\psi_{pq}: \cC\times_{|L_{pq}|}\cC\dashrightarrow \overline{\cM_{g,2}} $$
defined by mapping $(C',p,q)$ to the curve $C'$ with marked points $p,q$, analogously to the map \eqref{eq:mapMg2} defined for $L$. To compute the expected dimension of $Y_{pq}$ we can thus repeat the same computation used in the proof of Theorem \ref{thm:dimZn} to compute the expected dimension of $Z'_n(L)$. We can assume that $\psi_{pq}$ is generically quasi-finite for a general $(p,q)$, otherwise $\psi$ would not be generically quasi-finite, contrary to our assumptions in Theorem \ref{thm:dimZn}. Then, the expected dimension of $Y_{pq}$ is thus given by 
\begin{equation}
    \dim \cH_n + \dim ( \cC\times_{|L_{pq}|}\cC) - \dim \overline{\cM_{g,2}},
\end{equation}
since all these varieties have only quotients singularities. Now as in Theorem \ref{thm:dimZn}, $\cH_n$ has dimension $2g-1$ and the dimension of $\overline{\cM_{g,2}}$ is $3g-1$. But now the dimension of $\cC\times_{|L_{pq}|}\cC$ is $g$, since $L$ is very ample and so $\dim |L_{pq}|=\dim |L|-2=g-2$. We conclude that the expected dimension of $Y_{pq}$ is $0$. This prove that the expected dimension of $Z_n(L)$ is not smaller and so equal to the expected dimension of $Z'_n(L)$.
\end{proof}
This concludes the proof of the first part of Theorem \ref{thm:main1}.
%Notice that a priori $Z_n(L)$ can have different dimension than $Z'_n(L)$ if the map $\psi$ is not generically finite. For instance, this happens if for a general $(p,q)\in Z_n(L)$, the linear system of curves $|C'-p-q|$ for $C'\in |L|$ admits a positive dimensional subvariety $Y$ of curves $C_t$ with the property that $n(p-q)=0\in JC_t$. Nonetheless, it is reasonable it expect this to happen rarely. 
%One reason is that as already remarked in a more general setting (see remark \ref{rem:finproj}), this can't happen when $Z'_n(L)$ has finite projections. In this case both $\psi_*$ restricted to $Z_n(L)$ and $p_{i,Z_n(L)}$ must be finite. 
%We will see later in this section that we just have one different possibility than this, which is not expected in many examples. 

In the paper, we look for examples acting on the group $\ccc(X)$ generated by constant cycle curves (Definition \ref{def:corrccc}). For this, it is very important to find non-empty $Z_n(L)$ with projections satisfying the assumptions of Lemma \ref{lem:Z_*fult}. This properties will be discussed in the next subsections.

\subsection{Properties of the correspondence $Z_n(L)$.} We will now see how the correspondences $Z_n(L)$ have significant applications to the study of $CH_0(X)$ and $Z_1(X)$ and in particular to the study of the group $\ccc(X)$ generated by constant cycle curves. In this section, we assume that $Z_n(L)$ has the expected dimension, i.e. it is a $2$ cycle.

The following lemma is the key special property satisfied by $Z_n(L)$, from which all other properties follow.  
\begin{lemma}\label{lem:RE} For any point $(p,q) \in Z_n(L)$, $p$ is rationally equivalent to $q$ on $X$. 
\end{lemma}
\begin{proof} It is enough to prove the statement for $Z^0_n(L)$. Then it holds on the closure $Z_n(L)$ of $Z^0_n(L)$ by specialization to points of the closure. Indeed, for any point $(p,q)$ in the closure there is a dvr $R$ together with a map $\Spec R\to X\times X$ mapping the generic point $\eta$ of $\Spec R$ to the interior $(p_\eta ,q_\eta)\in Z^0_n(L)$ and the closed point $\Spec k$ to $(p,q)$. The family $X\times_{\Spec R} \Spec R\to \Spec R$ induces a specialization map $CH_0(X\times_{\Spec R} \eta/\eta)\to CH_0(X\times_{\Spec R} \Spec k/\Spec k)$. Assume the lemma holds on the interior. Then $p_\eta-q_\eta$ is a zero cycle rationally equivalent to $0$ on $X\times_{\Spec R} \eta/\eta$. By using the constructed specialization map $p-q$ is a cycle rationally equivalent to $0$ on $X\times_{\Spec R} \Spec k/\Spec k$ and so on $X$. For more details see for instance \cite[20.3]{Fu}.

To prove the statement for $Z^0_n(L)$, let $(p,q)\in Z_n^0(L)$. By Definition \ref{def:psiZ_n}, there is $C\subseteq X$ such that $n[p-q]=0\in JC$. Take the image of $n[p-q]=0$ via $JC\to CH_0(X)$. Since $CH_0(X)$ is torsion free, we must have that $[p-q]=0\in CH_0(X)$, which means that $p$ and $q$ are rationally equivalent on $X$.
\end{proof}

As an application, we show that the correspondences $Z_n(L)$ act on the group $\ccc(X)$ generated by constant cycle curves. 

Since $Z_n(X)$ is a $2$-cycle which is symmetric with respect to the involution of $X\times X$ interchanging the two factors, the projections $\pi_{i, Z_n(L)}:Z_n(L)\to X$, for $i=1,2$ are either both generically finite or they map onto the same curve $C$.  

Thus, by Lemma \ref{lem:Z_*fult} the homomorphism ${Z_n(L)}_*:CH_1(X)\to CH_1(X)$ lifts to a homomorphism of cycles 
\begin{equation}\label{mor:pushJ} {Z_n(L)}_*: Z_1(X)\to Z_1(X)
\end{equation}
and we can ask whether it preserves the group $\ccc(X)$.

\begin{theorem}\label{thm:Znccc} Assume that $Z_n(L)$ has dimension $2$. Then $Z_n(L)$ acts on the group $\ccc(X)$ generated by constant cycle curves (Definition \ref{def:corrccc}). Namely, ${Z_n(L)}_*:CH_1(X)\to CH_1(X)$ induces a homomorphism $$Z_n(L)_*:\ccc(X)\to \ccc(X).$$ 
%Furthermore, for any constant cycle curve $C$, $Z_*C$ is a constant cycle curve when not $0$. 
\end{theorem}
\begin{proof} Assume that $Z_n(L)$ has dimension $2$. As observed above, the projections satisfy Lemma \ref{lem:Z_*fult}, and so we have a well defined homomorphism $Z_n(L)_*: Z_1(X)\to Z_1(X)$. 

Thus, all we have to prove is that for any irreducible constant cycle curve $C\in \ccc(X)$, $C'=Z_n(L)_*C\in \ccc(X)$, i.e. it is a linear combination of irreducible constant cycle curves. Then the statement follows by linearity of $Z_n(L)$. 

Assume that $C'$ is not $0$, otherwise there is nothing to prove. By definition of $Z_n(L)_*$, $C'=\sum_i \alpha_i C'_i$, where any $C'_i$ is a reduced irreducible curve and $\alpha_i$ is the multiplicity along $C'_i$. To prove that $C'\in \ccc(X)$, it is enough to show that $C'_i$ is a constant cycle curve, for any $i$. For this, we need to show that two general points $q,q'\in C'_i$ are rationally equivalent. By Lemma \ref{lem:Z_*fult}, $C'=Z_n(L)_*C\in Z_1(p_{2,Z_n(L)}(|Z_n(L)|\cap p_{1,Z_n(L)}^{-1}(|C|)))$ and so $q,q'\in p_{2,Z_n(L)}(|Z_n(L)|\cap p_{1,Z_n(L)}^{-1}(|C|)).$ 
Then there exist $p,p'\in C$ (not necessarily distinct) such that $(p,q),(p',q')\in |Z_n(L)|\cap p_{1,Z_n(L)}^{-1}(C).$ 
Now since $p,p'\in C$ and $C$ is a constant cycle curve, $p$ and $p'$ are rationally equivalent on $X$. By Lemma \ref{lem:RE} we conclude that $q$ and $q'$ are rationally equivalent on $X$.

\end{proof}
Theorem \ref{thm:Znccc} proves the second part of Theorem \ref{thm:main1}.
 
 A last application concerns the relation with the generalized Bloch conjecture (see Subsection \ref{subsec:BB}). More precisely we can find a non-zero integer $d$ such that $Z_n(L)-d\Delta_X$, where $\Delta_X\subseteq X\times X$ is the diagonal, satisfies the so called easy-direction of the generalized Bloch conjecture (see Lemma \ref{lem:Voi11.18}, or directly \cite[proposition $11.18$]{Voi2}). In other words $Z_n(L)$ behaves as a multiple of the diagonal on zero cycles.

\begin{theorem}\label{thm:ZnBB} Assume $Z_n(L)$ has dimension $2$. Then $Z_n(L)_*$ acts by multiplication on $CH_0(X)$. In particular, $[Z_n(L)]^*$ acts by multiplication on  $H^0(X, \Omega^2_X)$ and $( Z_n(L))^*$ acts by multiplication on $H^2(X, \bQ)_{\tr}$.   
\end{theorem}
\begin{proof} Let $d=\deg \pi_{2,\psi_*Z_n}$. We claim that for any $p\in X$, $Z_n(L)_*([p])= d[p].$ Indeed, under our assumptions $(Z_n(L))_*([p])=[q_1]+[q_2]+\cdots [q_d]$, i.e. sum of $d$ points not necessarily distinct. By Lemma \ref{lem:RE}, $p$ and $q_i$ are rationally equivalent. This concludes the proof of the claim. 

Define $\Gamma = Z_n(L)-d \Delta_X,$ where $\Delta_X$ denotes the diagonal of $X\times X$. We conclude from above that $\Gamma_*=0$ on $CH_0(X)$. Now notice that $X$ is K3 surface and so $\Alb(X)$ is $0$, the map $\alb_X$ is $0$ and $F^2CH_0(X)=CH_0(X)_{\homo}$ (as introduced in Subsection \ref{subsec:BB}). We can then apply \cite[proposition $11.18$]{Voi2} and \cite[remark 11.20]{Voi2} (see directly Subsection \ref{subsec:BB}, Lemma \ref{lem:Voi11.18}) and conclude that $[\Gamma]^*=0$ on $H^0(X, \Omega^2_X)$, which proves the statement.
\end{proof}

\section{Correspondences on K3 surfaces of degree $2$.}\label{sec:Z2}
In this section we construct a natural locus $Z'_2(L)$ as in Definition \ref{def:Zn}, or more precisely $Z_2(L)$ as in Definition \ref{def:psiZ_n}, but for a line bundle $L$ of genus $2$, which is not very ample. 
%The locus is constructed on a K3 surface obtained as a double cover of $\bP^2$ ramified over a sextic curve. 
This provides an example with non-generically finite projections and we will see how this reflects on the study of constant cycle curves. This will hopefully clarify why we are particularly interested in having finite projections. 

\begin{theorem}\label{thm:Z2} Let $\pi:X\to \bP^2$ be a double cover of $\bP^2$ ramified over a sextic $D$. Consider the linear system $|L|$ defined by $L=\cO_X(C)$, where $C=\pi^*l$ is the pullback of a line $l$ of $\bP^2$. 
Then $Z_2(L)$ is not empty of dimension two, isomorphic to $D\times D$ up to base change and the projections $\pi_{i,Z_2(L)}:Z_2(L)\to X$ map onto $D$ with all fibers isomorphic to $D$. 
\end{theorem}
\begin{proof} Recall that a point $(C',p,q)\in Z'_2(L)$, with $C'\in |L|$ and $p,q\in C'$, satisfies $2[p-q]=0\in JC'$. On the one hand, $2[p-q]=0\in JC'$ for $C'\in |L|$ induces a degree $2$ covering totally ramified at $p$ and $q$, i.e. a $g^1_2.$ On the other hand, by definition $C'=\pi^*l'$ for some line $l'\subseteq \bP^2$, so the restriction of $\pi:X\to \bP^1$ to $C'\to l'$ is a degree $2$ cover ramified in 6 points of the sextic $D$, defining a $g^1_2$ on $C'$. These two degree $2$ coverings must coincide up to an isomorphism preserving the ramification locus, because the $g^1_2$ of a hyperelliptic curve is unique. Thus $p,q$ lie on $D$. Conversely, any pair of points $p,q$ on $D$ satisfy $2[p-q]=0\in JC'$ for $C'=\pi^*l'$ and $l'$ the line of $\bP^2$ through the two unique images $p',q'$ on $D$ of the points $p,q$, since $D$ is the ramification locus of $\pi$ and $\pi$ is totally ramified on $D$. Since $p'$ and $q'$ are linearly equivalent on the sextic by definition, their pullbacks $2p,2q$ are linearly equivalent on $C'$, which means that $2[p-q]=0\in JC'$. It follows that $Z'_2(L)$ is the locus of $(D,p,q)$ such that $p,q\in D$ and so it is isomorphic to a isotrivial family constructed from $D\times D$ by using the monodromy action. Since $D$ does not vary, the same holds for $Z_2(L)$. Now the projections have clearly image $D$ and all fibers $D$.
\end{proof}
It has been proven by \cite{Huy14} that the sextic $D\subseteq \bP^2$ defines a constant cycle curve in both $X$ and $\bP^2$. We can now use $Z_2$ to prove this fact again.
\begin{theorem}\label{thm:Dccc}
The ramification locus $D$ of a K3 $\pi:X\to \bP^2$ is a constant cycle curve in both $X$ and $\bP^2$.
\end{theorem}
\begin{proof} The proof is a corollary of Lemma \ref{lem:RE}, which doesn't require the assumption on $L$ to be very ample. For any $p\in D$,  $Z_2(L)_*(p)\simeq D$ by Theorem \ref{thm:Z2} and we conclude by Lemma \ref{lem:RE} that all points in $D$ are linearly equivalent, which means that $D$ is a constant cycle curve on $X$. By looking at the map $CH^1(X)\to CH^1(\bP^2)$ induced by $\pi$ we conclude that $D$ is a constant cycle curve also on $\bP^2$. 
\end{proof}
\begin{remark} It is clear that the only constant cycle curves that can be investigated by $Z_n(L)$ with projections mapping to a curve are the ramification locus and the fibers. Thus $Z_2(L)$ does not provide any new information to the study of constant cycle curves on this K3 surface.
\end{remark}

%\section{A correspondence $J$ from lines intersecting a point with multiplicity 3 in a quartic surface and $Z_3(H)$ of a hyperplane section.}\label{sec:J}
\section{Correspondences on K3 surfaces of degree $4$.}\label{sec:J}
In this section we construct and study a correspondence $J$ on a smooth quartic surface $X$ over $\bC$, by detecting lines intersecting $X$ in at most two distinct points, one of them with multiplicity at least $3$. We use $J$ to prove non-emptiness of $Z_3(H)$ associated to the linear system $|H|$ of a hyperplane section (as in Definitions \ref{def:Zn}, \ref{def:psiZ_n}). We will see that there is a very nice geometric relation between $J$ and $Z_3(H)$ given by $g_1^3$s defined by points of $J$ and $3$-torsion points. 

\subsection{Construction and properties of the correspondence $J$.}\label{subsec:conJ} Let $X\subseteq \bP^3$ be a smooth quartic surface over $\bC$. For any pair of distinct points $p,q$ of $X$ we denote by $l_{p,q}$ the line in $\bP^3$ through $p$ and $q$ and by $\Delta_X \subseteq X\times X$ the diagonal. Since $X$ is a smooth quartic surface, it contains at most finitely many lines (see e.g. \cite{GR16}), and outside of them the intersection $C\cdot l_{p,q}$ defines a $0$-cycle of degree $4$ in $X,$ i.e. an element in $Z_0(X).$ 
\begin{theorem}\label{thm:corrJ} Let $X\subseteq \bP^3$ be a smooth quartic surface over $\bC$. Then the closure $J\subseteq X\times X$ of 
\begin{equation}
    J^0= \{(p,q)\in X\times X\setminus \Delta_X \,|\, X\cdot l_{p,q}=3p+q \in Z_0(X)\}
\end{equation}
is a $2$ dimensional cycle of $X\times X$ yielding a correspondence whose first projection $\pi_{1,J}:J\to X$ is generically finite of degree $2$. Moreover, if we assume that $X$ does not contain a line, the second projection $\pi_{2,J}:J\to X$ has degree $68$. 
%Furthermore, for the general smooth quartic surface $X$ the second projection is finite.  
\end{theorem}
\begin{remark}\label{remLines} Later in the paper we will provide an alternative proof to the finiteness of the second projection which holds without the assumption that $X$ does not contain a line. The computation of the degree is not recovered by this alternative proof.
\end{remark}

\begin{proof} 
We first prove that $\pi_{1,J}$ is generically finite dominant of degree $2$. It follows from this that $J$ is $2$ dimensional.

Notice that as observed above $X$ contains at most finitely many lines (see e.g. \cite{GR16}), and therefore it is enough to work with $J^0$. Let $p\in X$ be a general point. The fiber $\pi_{1,J}^{-1}(p)$ of $p$ is parameterized by all lines in $\bP^3$ intersecting $p$ with multiplicity at least $3$. Any of these lines will in fact intersect $X$ in exactly one more point $q$ (not necessarily distinct from $p$), defining a point $(p,q)$ of the fiber. %Conversely, the line through a point $(p,q)\in \pi_{1,J}^{-1}(p)$ intersects $p$ with multiplicity $3$ by definition. 
Now notice that any line intersecting $p$ with multiplicity at least $2$ must lie in the tangent space $T_pX$. Therefore we have reduced the problem to counting lines in $T_pX$ intersecting $p$ with multiplicity $3$.

Consider the quartic plane curve $Q_p=T_{p,X}\cap X\subseteq T_{p,X}$ cut out by the tangent space $T_{p,X}$. By definition the curve $Q_p$ is singular at $p$. We know that any line in $T_pX$ through $p$ intersects $p$ with multiplicity at least $2$. Among them, the lines determining the tangent cone of $Q_p$ at $p$ are exactly those intersecting $p$ with multiplicity at least $3$. In conclusion, we need to determine the number of lines generating the tangent cone, counted without multiplicity. A double line will in fact determine only one point on the fiber $\pi_{1,J}^{-1}(p)$. 

Now for a general $p$, $Q_p$ is irreducible and $p$ is a node of $Q_p$. Thus the tangent cone is generated by two distinct lines $l_1,l_2$, each of them intersecting $X$ in exactly one more point different from $p$. We obtain two points $q_1,q_2$ that must be distinct by construction and so the fiber $\pi_{1,J}^{-1}(p)$ is given by two points $(p,q_1)$,$(p,q_2)$.

%by a classification theorem (see e.g. \cite[Lemma 1.1.6]{W14}) $p$ is either a node, a cusp or a tacnode (according to the definition given in \cite{W14} which is a bit unconventional) or a point of multiplicity $3$. In all cases the tangent cone is generated by either one or two distinguished lines (see equations in \cite[Figure 1.1.1]{W14}). 
%Since $X$ doesn't contain lines, reducible $Q_p$ must be given by two conics intersecting at $p$ so again the tangent cone is generated by either one or two distinguished lines. 
%In the second case the tangent cone is generated by one double line interesting $X$ is a point $q$ different than $p$, so the fibre $\pi_1^{-1}(p)$ consists of a point $(p,q)$. In the last case the tangent cone is again generated by a double line but this time the line intersects $X$ only at $p$ with multiplicity $4$. So the point $p$ must in this case lie in the boundary $J\setminus J^0$ and the fibre $\pi_{1,J}^{-1}(p)$ is given by the point $(p,p)$ (and conversely this description shows that tacnodes parametize the boundary). 
This proves that $\pi_{1,J}$ is generically finite of degree $2$.
%By the analysis above $\pi_{1,J}$ is finite and since the general point $p$ of a smooth quartic surface $X$ not containing a line defines a nodal point $p$ of the quartic plane curve $Q_p$, we conclude that $\pi_{1,J}$ has degree $2$. 

%for a general $p$ this is in fact the only singularity and it is a node. One can see this by looking at the equations. Let $f=f^4+x_0^1f^3+x_0^2f^2+x_0^3f^1+x_0^4f^0$ be the equation of $X$ with respect to the homogeneous coordinate $x_0, \dots , x_4$ of $\bP^3$, where $f^i$ are homogeneous of degree $i$ in $x_1, x_2, x_3$. The condition of passing through $p$ implies $f^0=0$ and the tangency that $f^1=0$. Up to a change of coordinates $x_1$ to $x'_1=f^1$, there is a homogeneous polynomial $f^{'2}$ of degree $2$ not containing $x'_1$ and so a polynomial of degree $2$ in only two variables vanishing on the plane $f^1=x'_1=0$. Its locus of zeroes for a general choice of $p$ is given by two distinct lines $l_1, l_2$ that meet transversely at $p$. 
%In other words, $p$ is a node of $Q_p$ with tangent cone in $Q_p$ given by lines $l_1, l_2\subseteq T_{p,X}$. 

%Now let us note that any line $l$ through $p$ has multiplicity at least $2$ at $p$ because $p$ is a node, and more precisely the general line intersects the node transversely and thus has multiplicity exactly $2$ while the tangent lines $l_1, l_2$ intersect the node with multiplicity one more, that is $3$. 
%Therefore, the only lines such that $l\cdot X=3p+q$ are exactly $l_1, l_2$ and we set $l_i\cdot X=3p+q_i$, for $i=1,2$. This concludes the first part of the proof. 

\vspace{1.5pt}
We now prove that the second projection $\pi_{2,J}$ is generically finite of degree $68$. Let $q$ be a point of $X$ and consider the projection $\pi_q: X\setminus \{q\}\to \bP^2$ from the point $q$, which is a finite rational map on $X$. The blow up $\epsilon: X'\to X$ of $X$ at $q$ resolves the indeterminacy of $\pi_q$ and gives a morphism $\pi'_q: X'\to \bP^2$, mapping the exceptional divisor $E$ to a line $L$ in $\bP^2$. Recall that the projection is defined for a choice of a suitable hyperplane $H$ not containing $q$ (which is then abstractly identified with $\bP^2$) by taking for any point $p\neq q$ of $X$ the line $l_{p,q}$ through $p,q$ and sending $p$ to the intersection point of $l_{p,q}\cdot H$. Since $X$ is a quartic surface, the intersection of $X$ with any line through $q$ is four points counted with multiplicity, one of each is $q$ itself. Thus the morphism $\pi'_q$ has degree $3$ outside the exceptional divisor. The points of total ramification of $\pi'_q$ are exactly the points $p\in X$ such that $l_{pq}\cdot X=3p+q$. Therefore, the sublocus of points of total ramification in the ramification locus $R$ of $\pi'_q$ coincides with the fiber $\pi_{2,J}^{-1}(q)$ of $\pi_{2,J}$ over $q$. We now compute such a locus.

Denote by $B$ the branch locus of $\pi'_q$. Since the degree of the morphism $\pi'_q$ outside the exceptional divisor is 3, the points of ramification are either simple, i.e. exactly two branches come together, or of total ramification, i.e. all branches come together. In particular, either the general point of ramification is simple or all points of ramification have total ramification and the cover is Galois. By \cite[Theorem 1]{Y01} the set of points $q$ corresponding to morphisms $\pi'_q$ that come from Galois covers is finite and by \cite[Corollary 2.2]{Y01} described by the points $q$ such that $X\cap T_qX$ is given by $4$ distinct lines. Since by assumption $X$ does not contain lines, such a set is empty. Therefore, for all points $q\in X$, the morphism $\pi_q'$ is such that the general point of $R$ has simple ramification. Because $\pi_q'$ has degree $3$ outside the exceptional divisor, we can describe ${\pi'_q}^{-1}(B)$ as the union of two curves. One, which is the ramification locus $R$ and whose general point has multiplicity $2$ at the intersection with each line through $q$. The other, denoted by $A$, is cut out by the last intersection point of each line with $X'$. The intersection $R\cdot A$ is a zero cycle whose points are exactly the points of $R$ having non-simple ramification, i.e. having more than two branches coming together. These points are exactly the points $p\in X$ such that $q+3p=l_{p,q}\cdot X$. Therefore, for a general $q$, they are the points in the fiber of the second projection at $q$. We then conclude that the degree of the second projection $\pi_{2,J}$ is exactly the degree of the cycle $A\cdot R$. 

We now compute the degree of $A\cdot R$. Let us first compute the ramification divisor $R$. Since $X$ is a $K3$, the canonical divisor $K_X$ is trivial and so the canonical divisor of the blow up $\epsilon:X'\to X$ at $q$ with exceptional divisor $E$ is just  $K_{X'}=K_X+E=E$. Let $L$ be the line such that $\pi'_q(E)=L$ and let $H'={\pi'_q}^*L$. If $H$ denotes a hyperplane of $X$, and $\epsilon: X'\to X$ is the blow up as above, then $\epsilon^*H=H'+E$.  We have $\epsilon^*H\cdot E=0$, that is $(H'+E)\cdot E=0$, $E^2=-1$, then $H'\cdot E=1$ and $\epsilon^*H\cdot \epsilon^*H=(H'+E)\cdot(H'+E)=4$. We conclude from this that ${H'}^2=3$. By the Hurwitz formula, $K_{X'}={\pi'_q}^*K_{\bP^2}+R$. Since $K_X'=E$ and ${\pi'_q}^*\cO_{\bP^2}(1)=H'$, we conclude that $R=3H'+E$. We now compute $A$. Since ${\pi'_q}_*H'=3L$, ${\pi'_q}_*E=L$, we obtain ${\pi'_q}_*R=10L$ and so $B=10L$. But now, $A+R={\pi'_q}^*B=10H'$ and since $R=3H'+E$ we conclude that $A=7H'-E$. To conclude, $A\cdot R= (7H'-E)\cdot (3H'+E)=68$.

The above argument shows that for all $q\in X$ such that $\pi_{2,J}^{-1}(q)\subseteq J^0$, the projection $\pi_{2,J}$ is finite of degree $68$.
\end{proof}

\begin{remark} A classification theorem by \cite[Lemma 1.1.6]{W14}, describes the singularity of $p$ as a point of the quartic curve cut out by its tangent space on $X$, for any point $p$ of $X$. More precisely,  $p$ is either a node, a cusp or a tacnode (according to the definition given in \cite{W14} which is a bit unconventional) or a point of multiplicity $3$. In all these cases the tangent cone is indeed generated by either one or two distinguished lines (see equations in \cite[Figure 1.1.1]{W14}). It seems that with a little attention the closure of $J$ could be fully explicitly understood in terms of this result, as well as the finiteness of the first projection.   
\end{remark}

%\begin{remark}\label{rem:ram} Note that in the proof above we have described the fibre over a point $q\in X$ of the second projection $\pi_2$ as the locus of points of total ramification of projection $\pi_q$ defined by $q$, which is a rational map of degree $3$ from the surface. For the general projection such a locus does not coincide with the ramification divisor and it  is cut out by the intersection with another curve, so it is a bunch of points. In this case special $q\in X$ means that $q$ defines a projection $\pi_q$ where any ramification point has total ramification. In this case the fibre of $\pi_2$ over $q$ coincides with the ramification curve $R$ and so it is one dimensional. This fact has important applications to the study of constant cycle curves, as we will see in the following subsections. 
%\end{remark}
%It is a special case of the pushforward property for cycles, valid on any adequate equivalence relations such is rational equivalence, that for any cycle and $\alpha \in Z^{*}(X)$ such that the introduced cycle  $J \in Z^{2}(X\times X)$ intersects the product $\alpha\times X$ properly, if $\alpha$ is rationally equivalent to $0$, then the same hold for ${\pi_{2,J}}_*(J \cdot (\alpha \times X)) $. 
A very important point in the study of the correspondence $J$ is that it satisfies a "special property" analogous to the that of $Z_n(L)$ (Lemma \ref{lem:RE}).  

\begin{lemma}\label{lem:clue} Let $(p,q), (p',q')\in J$. Then $p$ is rationally equivalent to $p'$ on $X$ if and only if $q$ is rationally equivalent to $q'$ on $X$.  
\end{lemma}
\begin{proof}
As in Lemma \ref{lem:RE}, if we prove the statement for $(p,q), (p',q')\in J^0$, then we conclude for all points of $J$ by specializing to points of the closure.

Let $(p,q), (p',q')\in J^0$. Since the Grassmannian variety of lines in $\bP^3$ is rational, $X\cdot l_{p,q}$ is rationally equivalent to $X\cdot l_{p',q'}$. In other words, $3p+q$ is rationally equivalent to $3p'+q'$.

Assume that $p$ is rationally equivalent to $p'$, then the same holds for $3p$ and $3p'$. From above, $3p+q$ is rationally equivalent to $3p'+q'$. Thus, by difference, $q$ and $q'$ are rationally equivalent. 

Conversely, let $q$ be rationally equivalent to $q'$. Since from above $3p+q$ is rationally equivalent to $3p'+q'$, by difference the same holds for $3p$ and $3p'$. This means that $p-p'$ is a $3$ torsion cycle of $X$. Now since $CH_0(X)$ is torsion free (see \cite{Roj80} or also \cite[theorem 14.14]{Voi2}), $p$ and $p'$ must be rationally equivalent on $X$. 
\end{proof}
\begin{remark}\label{rem:4cx} Alternatively, to prove the previous lemma one can compute the class of $l_{p,q}\cap X$, for a general pair of distinct $p,q\in X$. Since the line $l_{p,q}$ is intersection of a general  hyperplane through $p$ and a general hyperplane section through $q$, this class must be a multiple of the Beauville--Voisin class $c_X$ (see \cite{BV04}), and since $X$ is a quartic surface, the it must be $[l_{p,q}\cdot X]=4c_X$. Thus, in particular, any two lines through a pair of points of $X$ are rationally equivalent.
\end{remark}

 Using the previous lemma we can provide a proof to the generic finiteness of the second projection $\pi_{2,J}$ of Theorem \ref{thm:corrJ}, which does not compute the degree but it does not require the assumption that $X$ does not contain any line. 
 
 \begin{proof}[{\it Alternative proof to the generic finiteness of $\pi_{2,J}$}, which holds for any smooth quartic surface $X$. ] Assume by contradiction that $\pi_{2,J}$ is not generically finite and so the image $\pi_{2,J}(J)$ is a curve $C$ since it is not constant. Using that $\pi_{1,J}$ is finite, we can show that there exists a curve $C''\subseteq X$ such that $CH_0(C'')\to CH_0(X)$ is surjective. This leads to a contradiction because the Chow group of a curve is finitely generated but by Mumford's Theorem (see \cite{Mum68}) the Chow group of $X$ is not. 
 
 In other words, we prove that there exists a curve $C''\subseteq X$ such that for any point $x\in X$ there is a point $p\in C$ such that $x$ and $p$ are rationally equivalent.
 
 Since $\pi_{2,J}:J\to C$ is surjective by assumption and $J$ is two dimensional, there exists a curve $C'\subseteq J$ such that the restriction $\pi_{2,J}:C'\to C$ is still surjective. We claim that its image with respect to the first projection $\pi_{1,J}(C')$ is a curve $C''$ with the property above. 
 
 Let $s\in X$ be a general point, then since $\pi_{1,J}$ is dominant there exists $q\in X$ such that $(s,q)\in J$ and $\pi_{1,J}((s,q))=s$. By construction $\pi_{2,J}((s,q))=q\in C$ and so using that $\pi_{2,J}:C'\to C$ is surjective there exists also $p'\in X$ such that $(p',q)\in C'\subseteq  J$. Since $(p', q)\in C',$ $\pi_{1,J}((p',q))=p'\in \pi_{1,J}(C')=C''$ and so for a general $s\in X$ we have found a rationally equivalent point $p'\in C''$ as claimed.  
\end{proof}

 As an application we recover a relation with the generalized Bloch conjecture (see Subsection \ref{subsec:BB}). More precisely, we show that $J-6\Delta_X$, where $\Delta_X\subseteq X\times X$ is the diagonal, satisfies the inverse direction of the generalized Bloch conjecture (see Lemma \ref{lem:Voi11.18}, or directly \cite[proposition $11.18$]{Voi2}). 
 
\begin{theorem}\label{thm:JBB} Let $\Delta_X\subseteq X\times X$ denote the diagonal. The relation $J_*-6{\Delta_X}_*=0$ holds on $CH_0(X)_{\homo}$, and so  $J^*-6\Delta_X^*=0$ on $H^2(X, \bQ)_{\tr}$.   
\end{theorem}
\begin{proof} For a general $p\in X$ , $J_*([p])= [q_1]+[q_2]$ by definition. By Lemma \ref{lem:clue}, $q_1$ and $q_2$ are rationally equivalent and we obtain $J_*([p])= 2[q_1]$. Now notice that the class $[X\cdot l_{pq_1}]$ is the class of the intersection of $X$ with the  line $l_{p,q_1}$, which is equivalently  realized as the intersection of any hyperplane through $p$ and a hyperplane through $q_1$. As explained in Remark \ref{rem:4cx}, this class must be $4c_X$, where $c_X$ is the Beauville--Voisin class. Equivalently, it must be $4[t]$, where $t$ is a point of a rational curve. Consequently, we can write $[3p+q_1]=[X\cdot l_{pq_1}]= 4[t]$. Then $J_*([p])=2[q]= 2[4t-3p]=8[t]+6[p]$. In particular, on $CH_0(X)_{\homo}$ we have $J_*=6{\Delta_X}_*$. Now notice that $X$ is a K3 surface and so $\Alb(X)$ is $0$, the map $\alb_X$ is $0$ and $F^2CH_0(X)=CH_0(X)_{\homo}$ (as introduced in Subsection \ref{subsec:BB}). We can then apply \cite[proposition $11.18$]{Voi2} and \cite[remark 11.20]{Voi2} (see directly Subsection \ref{subsec:BB}, Lemma \ref{lem:Voi11.18}) to conclude the second part of the statement.
\end{proof}

\begin{corollary}For the generic smooth quartic surface $[J]=-4[x\times X]+62[X\times x]+6[\Delta_X]+d[h\cdot h^{-1}]$, for $x\in X$ and $h$ a curve in $X\times X$ that descends to $\cO_X(1)$.
\end{corollary}
\begin{proof} For a general quartic surface $[J]=a[x\times X]+b[X\times x]+c[\Delta_X]+d[h\cdot h^{-1}]$. From Theorem \ref{thm:JBB}, $c=6$. Now $\deg \pi_{1,J}=2$ and $\deg \pi_{2,J}=68$ so $a+c=2$ and $b+c=68$ and we conclude. 
\end{proof}

\subsection{Non-emptiness of $Z_3(H)$ defined by a hyperplane section.}
We relate $J$ to $Z_3(H)$, where $H$ is a hyperplane section of the smooth quartic $X\subseteq \mathbb P^3$.

Let $X$ be a smooth quartic surface that does not contain a line. Let $H\subseteq X$ be a hyperplane section of $X$ and let $\cC\to |H|$ be the universal curve of the linear system $|H|$, whose fiber over a general member $C\in |H|$ is the quartic curve $C$ itself. Recall that $Z_3'(H)\subseteq \cC\times_{|H|}\cC$ is the sublocus of $3$-torsion points of $|H|$ as defined in Section \ref{sec:Zn}. 

In these notations, 
\begin{equation}
    Z'_3(H)= \overline{\{(p,q, C)\in \cC\times_{|H|}\cC \, |\, p\neq q, \, 3[p-q]=0\in JC\}}\subseteq \cC\times_{|H|}\cC,
\end{equation}
 \begin{equation} Z_3(H)= \overline{\{(p,q)\in X\times X\, |\,  \exists \, C\in |H| \mbox{ smooth s.t. }p,q\in C, p\neq q, \, \,3[p-q]=0\in JC\}}\subseteq X\times X
 \end{equation}
 and they are related by the forgetful map $Z'_3(H)\to Z_3(H).$

As stressed in Section \ref{sec:Zn}, it is not clear in general whether $Z_3(H)$ is non-empty. We now show that $Z'_3(H)$ and $Z_3(H)$ are non-empty of dimension $2$. We use $J$, or more precisely the fiber product $J\times_{\pi_{2,J}(J) }J$.

Notice that a point in this fiber product is given by two points $(p,q), (p',q')\in J$ such that $q=q'$. We will thus write a point of $J\times_{\pi_{2,J}(J) }J$ as a tuple $((p,q), (p',q)).$ 
\begin{theorem}\label{thm:JtoZ}
There is a finite dominant rational map $J\times_{\pi_{2,J}(J) }J\dashrightarrow Z_3'(H)$. 
%The map sends a point $((p,q),(p',q))$ of $J\times_{\pi_{2,J}(J) }J$  to the point $(p,p',C')$ of $Z_3'(H)$, where $C'$ defined as the intersection $X\cap H_{p,p'}$ of $X$ with the hyperplane $H_{p,p'}^q\subseteq \bP^3$ generated by $p,p',q$.  
In particular, both $Z_3'(H)$ and $Z_3(H)$ have dimension $2$.
\end{theorem}
\begin{proof}
Recall that $J^0\subset J$ denotes the interior of $J$ (see Theorem \ref{thm:corrJ}). Let $\Delta \subseteq J\times_{\pi_{2,J}(J) }J$ denote the locus of points $((p,q), (p',q))$ such that $p=p'.$ Let $((p,q), (p',q))$ be a point of $J^0\times_{\pi_{2,J}(J)^0}J^0\setminus \Delta$. By definition of $J$ (Theorem \ref{thm:corrJ}), $3p+q=X\cdot  l_{p,q}$ and $3p'+q= X\cdot l_{p',q}.$ 

Consider the plane $H_{p,p'}^q\subseteq \bP^3$ through $p,p',q.$ Then $C=H_{p,p'}^q\cap X$ is a quartic containing the points $p,p'$. Furthermore, for a general $p$, the quartic $C$ is smooth in the linear system $|H|$. Indeed, a quartic is singular if it is cut out by the tangent space at a certain point of $X$. If we consider the lines through $p,q$ in the construction, for any $p$ we have at most two choices of $q$, so they are parametrized by $p$ up to a finite choice for $q$. We can fix a choice of $q$ for any $p$ so that we obtain a pencil of lines parametrized by $p$. We look at them as a line in ${\bP^3}^\vee$ and intersect it with the dual variety defined by the Gauss map. The dual variety parametrizes tangent spaces and has degree $4(3)^2$. This number is lower than the degree of the second projection, which is $68$ by Theorem \ref{thm:corrJ}. This means that we can choose at least some $p'$ such that the plane though $p,q,p'$ is not a tangent space. 

Now by definition of $J$, $3p+q=X\cdot  l_{p,q}$ and $3p'+q= X\cdot l_{p',q}.$ Furthermore $3p+q$ and $3p'+q$ are rationally equivalent, since the Grassmannian of lines of $\bP^3$ is rational. Consequently, $3p$ and $3p'$ are rationally equivalent on $X$.
Since the Grassmannian of lines in $\bP^2$ is also rational, by looking at the curve $C$ and the lines $l_{p,q}, l_{p'.q}$ in the hyperplane $H_{p.p'}^q\simeq \bP^2$ we get $[3p-3p']=0\in JC$ and so $3p$ and $3p'$ are also rationally equivalent on the smooth quartic $C$, as we wanted.

This defines a map $$U\setminus \Delta\to Z_3'(H),$$
from an open dense subset $U\subseteq J^0\times_{\pi_{2,J^0}(J^0) }J^0$.
The map sends a point $((p,q),(p',q))\in U$  to the point $(p,p',C)\in Z_3'(H)$, where $C$ defined as above. Namely, as the intersection $X\cap H_{p,p'}$ of $X$ with the hyperplane $H_{p,p'}^q\subseteq \bP^3$ generated by $p,p',q$.  

We now prove that such a map is surjective on $Z'^0_3(H)$. Let $(p,p',C)\in Z'^0_3(H)$, then the condition $3[p-p']=0\in JC$ defines a $g^3_1$ on $C$. Now it is well known that all the $g^3_1$ of quartic curves are obtained as projections from a point of the quartic (see \cite{ACGH85}). Namely, there exists $q\in C$ such that the $g^1_3$ is defined by the projection $\pi_q:C\setminus\{q\}\to \bP^1$ from $q$, which extends over $q$. By definition of $\pi_q$, $p$ and $p'$ are points of total ramification of the degree $3$ map $\pi_q:C\setminus\{q\}\to \bP^1$. Since the projection from a point $q$ is defined by intersecting with lines through $q$, we obtain $X\cdot l_{p,q}=3p+q$ and $X\cdot l_{p',q}=3p'+q$. In particular, $((p,q), (p',q))$ defines a point of $J^0\times_{\pi_{2,J^0}(J^0) }J^0$ as wanted.

We now prove that the rational map is finite. Fix a general point $(p,p',C')\in Z_3'(H)$. Then the fibers $\pi_{1,J}^{-1}(p)$, $\pi_{1,J}^{-1}(p')$ of $p,p'$ via the first projection of $J$ are finite. Thus we have only finitely many $q$ in the fiber over $(p,p',C')$. 
%Using the same argument as in the proof of theorem \ref{thm:corrJ}, we know that for any $q\in X$, the general point of ramification of  $\pi_q$ has simple ramification (see also directly \cite[Corollary 2.2]{Y01}). In other words, there exist only finitely many points $p_i$ of total ramification and so finitely many $(p_i,q)\in J^0$ in the fibre of $(p,q)\in Z_3'^0(H).$ Thus the rational map is finite.

Recall that $J$ has dimension $2$ and $\pi_{2,J}$ is generically finite. Therefore the fiber product $ J^0\times_{\pi_{2,J^0}(J^0) }J^0$ as well as $U\setminus \Delta$ are $2$-dimensional. Thus we conclude that also $Z'^0_3(H)$ is not empty of dimension $2$, as it is dominated by a variety of dimension $2$ via a generically finite map. 

A similar argument shows that $Z_3(H)$ is also not empty of dimension $2$. It is enough to replace the map $U\setminus \Delta\to Z_3'(H)$ with the analogous map $U\setminus \Delta\to Z_3(H)$, sending $((p,q),(p',q'))\in U\setminus \Delta$ to $(p,p')\in Z_3(H)$.
%For this, consider the forgetful map $Z_3'^0(H)\to Z_3^0(H)$ sending $(p,p',C)$ to $(p,p')$. We need to prove that there is not a positive dimensional locus of smooth curves $C$ through a general choice of $p,p'$ such that $3[p-q]=0\in JC$. Assume by contradiction that this is the case. The rational map constructed above gives for any $p,p'$ a precise curve $C$ and a point $q\in C$. The points $q_i$ for fixed $p,p'$ are finite and $C$ is uniquely recovered from them. Thus we have a two dimensional family $(p_t,p'_t)$ of pairs of points of $X$ defining an open dense set of points $(C_t,p_t,q_t)$ of $Z'_3(L)$. Consequently, if for a general choice of $(p,p')$ we had a positive dimensional locus of curves $C$, then $Z'_3(L)$ would have dimension $>2$, which gives a contradiction.
\end{proof}

\subsection{The correspondence $J$ and constant cycle curves.}
In this section, we study the correspondence $J$ in relation to constant cycle curves. First of all, notice that $J$ has generically finite projections by Theorem \ref{thm:corrJ}. Thus, by Lemma \ref{lem:Z_*fult} there is a well-defined homorphism at the level of cycles
\begin{equation}\label{mor:pushJ} J_*: Z_1(X)\to Z_1(X).
\end{equation}
We can ask whether this homomorphism preserves the subgroup $\ccc(X)$ defined by constant cycle curves (Definition \ref{def:setccc}).

%We recall that in our notations $X$ is a smooth quartic surface, $J\in Z_2(X\times X)$ is the correspondence introduced in Theorem \ref{thm:corrJ}, which corresponds to a point $p\in X$ the set of points $q\in X$ such that the intersection of $X$ with the line $l_{p,q}$ through $p,q$ is the cycle $3p+q$, and $\pi_1$ and $\pi_2$ are respectively the first and second projection.
\begin{proposition}\label{prop:ccc1} $J$ acts on the group $\ccc(X)$ generated by constant cycle curves (Definition \ref{def:corrccc}). In other words, $J_*$ as in \eqref{mor:pushJ} induces a homomorphism $$J_*:\ccc(X)\to \ccc(X).$$ 

%Furthermore, for any constant cycle curve  $C\subseteq X$, $J_*C$ is a constant cycle curve if not $0$.
\end{proposition}
\begin{proof} 
We proceed as in Lemma \ref{lem:RE}. More precisely, we prove that for any $C\in \ccc(X)$ irreducible constant cycle curve, $C'=J_*C$ is a linear combination of constant cycle curves if not $0$. From this it follows by linearity that $J$ acts on $\ccc(X)$. 

Assume  $C'$ is not $0$. By definition, $C'$ is of the form  $C'=\sum_i \alpha_i C'_i$, where any $C'_i$ is a irreducible reduced curve and $\alpha_i$ is the multiplicity along $C'_i$. We must show that every $C_i$ is a constant cycle curve.

For this, we need to show that two distinct points $q,q'\in C'_i$ are rationally equivalent on $X$ (Definition \ref{def:ccc}). By Lemma \ref{lem:Z_*fult}, $C'=J_*C\in Z_1(p_{2,J}(|J|\cap p_{1,J}^{-1}(C)))$ and so $q,q'\in p_{2,J}(|J|\cap p_{1,J}^{-1}(C)).$ 
Thus there exist $p,p'\in C$ (not necessarily distinct) such that $(p,q),(p',q')\in |J|\cap p_{1,J}^{-1}(C).$ 
Since $p,p'\in C$ and $C$ is a constant cycle curve, $p$ and $p'$ are rationally equivalent. By Lemma \ref{lem:clue} we conclude that $q$ and $q'$ are rationally equivalent. 

\end{proof}
%The first important application of the previous proposition concerns constant cycle curves defined by projections from special points of the quartic surface $X$.
%\begin{proposition}\label{prop:cccram} Let $R$ be the ramification curve of a projection $\pi_q:X\setminus \{q\}\to \bP^2$ be the projection from a point $q\in X$ thorough a constant cycle curve $C$. Assume that $\pi_q$ ramifies with total ramification at any point of $R$. Then $R$ is a constant cycle curve.   
%\end{proposition}
%\begin{proof} Let $q\in C\in X$ as in the assumptions of our proposition. Then by remark \ref{rem:ram} the fibre $\pi_2^{-1}(q)=R$ and so $R\subseteq \pi_2^*(C)$. Then by construction $\pi_{1*}\pi_2^*(C)= R+C'$ and by proposition \ref{prop:ccc1} $R+C'$ is a constant cycle curve so in particular $R$ is a constant cycle curve. 
%\end{proof}
%\begin{remark} By a result of Voisin contained in the appendix of \cite{Huy14} constant cycle curves are analytically dense on $X$ so it would be interesting to estimate how often I can find one through a point $q$ defining a projection as in the statement above. Are the union of these ramification curves Zariski dense?
%\end{remark}

As an application, we can construct constant cycle curves by applying $J_*$ to some known constant cycle curve, for instance a rational curve, and try to understand when the result is a new constant cycle curve. The first important step in this direction is to detect when the curve $J_*C$ does not coincide with the curve $C$.  
\begin{proposition}\label{prop:ccc2}
 Let $C\subseteq X$ be a curve such that $J_*C= C$. Then $C$ is either a planar curve of degree $4$ or has degree greater than $4$.  
\end{proposition}
\begin{proof}
Let $x\in \bP^3$ be a general point. By assumption $C=J_*C$, so there exist two points $p,q \in C$ such that $(p,q)\in J$ and $p, q$ both differ from $x$. Denote with $H$ the plane spanned by $x, p, q$. We have $C\cdot H=3p+q+R$, with $R$ a divisor of $C$. 

If $C$ is a planar curve in a quartic surface $X$, then $R$ must be empty and $C$ has degree $4$. Otherwise $C$ is not a planar. Consequently, for a general $x\in C$, there exists a non-empty $R$ such that  $C\cdot H=3p+q+x+R$. In particular, the degree of $C$ is greater than $4$. 
\end{proof}
In particular, since conics are rational curves, Propositions \ref{prop:ccc1} and \ref{prop:ccc2} give rise to new constant cycle curves in the following. 
\begin{corollary}\label{cor:conic}
Assume that $X$ contains a smooth conic $C$, then $J_*C$ is a constant cycle curve different from $C$. 
\end{corollary}
In particular, the corollary suggests how to construct non-rational constant cycle curves starting from conics in special position with respect to the branch divisor.
%As a corollary we produce examples in very special smooth quadric surfaces $X$ of non rational constant cycle curves constructed by rational curves.  
%\begin{corollary}
%Assume that a smooth quartic surface $X$ contains a smooth conic $Q$ s.t.
%\begin{enumerate}\item $Q$ is transversal to the branch locus of the first projection of $J$;
%\item the preimage of $Q$ via the $1$st projection is birational to the image via the $2$nd projection.
%\end{enumerate} 
%Then the homomorphism $J_*:\ccc (X)\to \ccc(X)$ maps $Q$ to a non rational constant cycle curve.
%\end{corollary}

%\begin{proof} Under our assumptions, the restriction of $\pi_{1,J}$ to $C'$ defines a degree $2$ cover from a smooth projecctive curve to a smooth conic, which abstractly isomorphic to the projective line. Thus $C'$ is hyperelliptic. By corollary \ref{cor:conic}, $C'$ is a constant cycle curve.  By assumptions, it is birational to a hyperelliptic curve $C'$ and so not rational. 
%\end{proof} 

\begin{remark} It is not so difficult to see that the branch curve of the projection $\pi_{1,J}$ coincides with the parabolic curve defined in \cite{W14}. It is conjectured in \cite[page 4]{W14} that this curve should be a constant cycle curve. It would be interesting to determine the validity of the conjecture by using $J$ and the properties stated above.
\end{remark}
%The importance of the previous statement might be contained in the following question.
%\begin{question}
%Let $X$ be a smooth quartic surface containing a smooth conic $C$. Is $C'$ rational? Otherwise, we would produce non rational constant cycle curves by using the correspondence $J$. 
%\end{question}

%We conclude this section with a remark relating $J$ with a certain constant cycle curve already constructed in \cite{Huy14}.
%\begin{remark}
%Notice that points $(p,p)\in J$ corresponds exaclty to the points where the two lines $l_1,l_2$ in the the proof of Theorem \ref{thm:corrJ} collapse to a double line $l_p$ such that $X\cdot l_p=4p$. The locus of these points coincides with the locus of hyperflexes, defined as the locus where bitangent lines to $X$ collapse and intersects in only one point of multiplicity $4$. In \cite{Huy14} it is proven that it defines a constant cycle curve. 
%\end{remark}

%\section{A correspondence $T$ from planes intersecting points with multiplicity $4$ on a complete intersection in $\bP^4$ and $Z_4(H)$ of a hyperplane section.}\label{sec:T}
\section{Correspondences on K3 surfaces of degree $6$.}\label{sec:T}
In this section we introduce a correspondence $T$ on $X\times X\times X,$ i.e. $2$-cycle of $X\times X\times X$ of dimension $2$, for a smooth complete intersection $X\subseteq \bP^4$ of a quadric and a cubic. The correspondence is defined by planes whose intersection with $X$ contains a point of multiplicity at least  $4$. This would be an attempt to generalize what we have done with lines in the correspondence $J$ of $X\times X$. We will actually see that we cannot immediately conclude that $T$ acts on the group $\ccc(X)$ generated by constant cycle curves from $X$. For instance, it is defined on $X\times X\times X$. Nonetheless, we can prove a result similar to Lemma \ref{lem:clue}, from which we can deduce non-emptiness of the locus $Z_4(H)$, associated to the linear system of a hyperplane section (as in Definitions \ref{def:Zn},\ref{def:psiZ_n}).

\subsection{Construction of the correspondence $T$.}\label{subsec:conT} Let $X\subseteq \bP^4$ be a smooth complete intersection of a quadric and a cubic defined over $\bC$. For any triple of distinct and non-collinear points $p,q,r$, let $H_{p,q,r}$ be the plane in $\bP^4$ generated by them. Denote by $\Delta_T \subseteq X\times X\times X$ the subset of triples $(p,q,r)$ that do not define a line.

\begin{theorem}\label{thm:corrT} Let $X\subseteq \bP^4$ be a general smooth complete intersection of a quadric and a cubic over $\bC$. Then the closure $T\subseteq X\times X\times X$ of 
\begin{equation}
    T^0= \{(p,q,r)\in X\times X\times X\setminus \Delta_T\,|\, X\cdot H^p_{q,r}=4p+q+r \in Z_0(X)\}
\end{equation}
is a $2$ dimensional cycle of $X\times X\times X$ yielding a correspondence from $X$ to $X\times X$ with generically finite projections. Furthermore, the first projection $\pi_{1,T}:T \to X$ has degree $2$.
\end{theorem}

\begin{proof} 
Notice that if $\pi_{1,T}$ is generically finite and dominant of degree $2$, then $T$ must have dimension $2$. 

We prove that $\pi_{1,T}$ is generically finite and dominant of degree $2$. For this, it is enough to show that for a general $p\in X$, there exist exactly two points $q,r$ of $X$ such that $X\cdot H^p_{q,r}=4p+q+r$. Then the two points $(q,r)$ and $(r,q)$ define the fiber $\pi_{1,T}^{-1}(p)$. 

Notice that the only plane that can satisfy the property defining $T$ is the tangent space $T_{p}X$. This can be checked with a local computation. 

%Consequently, we need to prove that for a general $p$, $T_{p}X$ intersects $X$ at $p$ in fact with multiplicity at least $4$. 

Recall that $X$ is a complete intersection of a quadric $A$ and a cubic $B$. For a general point $p$ of $X$, let $A_p=T_{p,X}\cap A$ and $B_p=T_{p,X}\cap B$ be respectively the quadric and the cubic plane curves cut out by the tangent space $T_{p,X}$ on $A$ and $B$. By definition they are both singular at $p$ and for a general choice of $p$ both the quadric and the cubic are nodal. In other words, $p$ has multiplicity $2$ on both the quadric and the surface. All together there are exactly two further points $q,r$ in the intersection between the cubic and the quadric. They will be distinct for a general choice of $p$. This proves that $T_pX \cdot X=4p+q+r$ and so the fiber of $\pi_{1,T}$ at $p$ is given by the two points $(q,r)$ and $(r,q)$. 

Let us prove that the projection $\pi_{23,T}$ is finite onto its image. Let $q,r$ be general points of $X$ and let $l_{qr}$ be the line through those points. Consider the projection $\pi_{qr}: X\setminus \{l_{qr}\}\to \bP^2$ from the line $l_{qr}$. Since $l_{qr}\cap X=\{q,r\}$ for $q,r\in X$ general, we deduce that the map $\pi_{qr}$ is a degree $4$ rational map. The ramification divisor is a curve and for a general projection the general ramification point is simple. Thus the points of total ramification define a finite (possibly empty) subset of points of the branch divisor. These are exactly the points of the fiber of $(q,r)\in X\times X,$ which is thus finite and non-empty since $T$ is not empty. 
This proves that $\pi_{23,T}$ is generically finite.
\end{proof}

%\begin{remark}\label{rem:finproj} We can also define projections $\pi_{i,T}: T\to X\times X\to X$, for $i=2,3,$ by composition of $\pi_{23,T}$ with the normal projections $\pi_{j,X}:X\times X\to X$, for $2,3$ (note that $2$ and $3$ indices here respectively the first and second projection).  In particular, they are generically finite as $\pi_{23,T}$. To prove this for instance for $\pi_{2,T}$,  we can compute the fibre over a point $r$. This is given by composition of the fibre of the projection $\pi_{23}$ and the first projection $\pi_{2,X} :X\times X\to X$, from the codomain of $\pi_{23}$. Now the fibre of $\pi_{23}$ is finite for the argument above, and the fibre of $\pi_{2,X}$ is parameterized by the points $q$ of $X$ such that the line through $r,q$ defines a projection $\pi_{qr}$ as above with a point of total ramification. Again we conclude because the projection from the general line does not have points of total ramifications.
%\end{remark}
%\begin{remark}
%Notice that as a consequence that $T$ has dimension $2$ and the map $\pi_{23}$ is finite, the projections from lines in $\bP^4$ through two points in the surface $X$ admitting points of total ramification is a condition in codimension $2$.
%\end{remark}

The natural generalization of Lemma \ref{lem:clue} given for $J$ and of Lemma \ref{lem:RE} given for $Z_n(L)$ is the following.

\begin{lemma}\label{lem:clueT} Let $(p,q,r), (p',q',r')\in T^0$. Then $p$ is rationally equivalent to $p'$ on $X$ if and only if $q+r$ is rationally equivalent to $q'+r'$ on $X\times X$. 
\end{lemma}
\begin{proof} Let $(p,q,r), (p',q',r')\in T^0$. Notice that by definition of $T^0$, $4p+q+r=X\cdot  T_pX$ and $4p'+q+r= X\cdot T_{p'}X$ and $4p+q+r$ and $4p'+q+r$ are rationally equivalent on $X$. In fact,  $T_pX$ and $T_{p'}X$ intersects in the line through $r$ and $q$ and define a pencil which is a $\bP^3$ intersecting $X$ on a curve $C$. The pencil restricted to $C$ contains both divisors $4p+q+r$ and $4p'+q+r$. 

%\rtext{Notice that by definition of $T$, $4p+q+r=X\cap  T_pX$ and $4p'+q+r= X\cap T_{p'}X$ and $4p+q+r$ and $4p'+q+r$ are rationally equivalent on $X$, since the Grassmannian of planes of $\bP^4$ is rational and we can move the two curves $D$ and $D'$ built by the three lines $l_{pq},l_{pr},l_{rq}$ and $l_{p'q},l_{p'r},l_{rq}$ respectively using rational parameters. }

Now let us assume that $p$ is rationally equivalent to $p'$. Then the same holds for $4p$ and $4p'$. Moreover, from above $4p+q+r$ is rationally equivalent to $4p'+q'+r'$ and by difference, we obtain the same for $q+r$ and $q'+r'$. 

Conversely, let $q+r$ be rationally equivalent to $q'+r'$. Using that $4p+q+r$ is rationally equivalent to $4p'+q'+r'$ we obtain the same for $4p$ and $4p'$ by difference. This means that $p-p'$ is a $4$ torsion cycle of $X$. Since $CH_0(X)$ is torsion free  (see for instance \cite[theorem 14.14]{Voi2}) then $p$ and $p'$ must be rationally equivalent.

\end{proof}
%Recalling that the push forward preserves rational equivalence, we immediately deduce the following.
%\begin{corollary}\label{cor:clueT} Let $(p,q,r), (p',q',r')\in T$ such that $p$ is rationally equivalent to $p'$ in $X$. Then $q$ is rationally equivalent to %$q'$ in $X$ and $r$ is rationally equivalent to $r'$ in $X$.
%\end{corollary}
%\begin{proof} By lemma \ref{lem:clueT} if $p$ and $p'$ are rationally equivalent in $X$ then $q+r$ is rationally equivalent to $q'+r'$ in $X\times X$. Now applying the projection $\pi_{2,T}$, we obtain $\pi_{2,T}_*(r+q)=$
%\end{proof}

\subsection{Non-emptiness of $Z_4(H)$ defined by a hyperplane section.}
The correspondence $T$ is strongly related to the existence of torsion points of order $4$ on the general member of the linear system given by a hyperplane section on $X$.

Let $H\subseteq X$ be a hyperplane section of $X$ and let $\cC\to |H|$ be the universal curve of the linear system $|H|$.  As defined in Section \ref{sec:Zn},  
\begin{equation}
    Z'_4(H)= \overline{\{(p,q, C')\in \cC\times_{|H|}\cC \, |\, p\neq q,\, 4[p-q]=0\in JC'\}}\subseteq \cC\times_{|H|}\cC,
\end{equation}
 \begin{equation} Z_4(H)= \overline{\{(p,q)\in X\times X\, |\,  \exists \, C'\in |H| \mbox{ smooth s.t. }p,q\in C', p\neq q, \,4[p-q]=0\in JC'\}}\subseteq X\times X
 \end{equation}
 and they are related by the forgetful map $Z'_4(H)\to Z_4(H).$

We show that $Z'_4(H)$ and $Z_4(H)$ are not empty. We use $T$, or more precisely, the fiber product $T\times_{\pi_{23}(T) }T$. Notice that a point in this fiber product is given by two points $(p,q,r), (p',q',r')\in T$ such that $q+r=q'+r'$ so we will write a point of $T\times_{\pi_{23}(T) }T$ just as a tuple $((p,q,r), (p',q,r)).$ 
\begin{theorem}\label{thm:JtoZ}
There is a finite rational map $T\times_{\pi_{23}(T) }T\dashrightarrow Z'_4(H)$.
%sending a point $((p,q,r), (p',q,r))$ of $T\times_{\pi_{23}(T) }T$ of the fibre of $q+r\in \pi_{23}(T)$ to the point $(p,p',C')$ of $Z_4$ with $C'$ given by the intersection of $X$ with the hyperplane $H_{p,p'}^{q,r}\subseteq \bP^3$ generated by the two tangent spaces $T_pX$ and $T_{p'}X$. 
In particular, $Z'_4(H)$ and $Z_4(H)$ are non-empty of dimension $2$. 
\end{theorem}
\begin{proof}
Recall that $T^0\subset T$ denotes the interior of $T$ (see Theorem \ref{thm:corrJ}). Let $\Delta_T \subseteq T\times_{\pi_{23}(T) }T$ denote the locus of points $((p',q,r), (p,q,r))$ such that $p=p'.$  
Let $((p,q,r), (p',q,r))$ be a point of $T^0\times_{\pi_{23}(T^0) }T^0\setminus \Delta_T$. Then by definition of $T^0$, $4p+q+r=X\cdot  T_pX$ and $4p'+q+r= X\cdot T_{p'}X.$ Consider the hyperplane $H_{p,p'}^{q,r}\subseteq \bP^4$ generated by $T_pX$ and $T_{p'}X$. Notice that $T_pX\cap T_p'X$ is the line generated by $q,r$. Consider the curve $C=H_{p,p'}^{q,r}\cap X$. For $X$ general, $C$ is smooth. Since such curve contains $4p+4p'+q+r$, we conclude that $4p+q+r$ and $4p'+q+r$ are rationally equivalent on $C$.

%Then $C=H_{p,p'}^{q,r}\cap X$ is a curve. For a general choice of points $p,p'$, $C$ is smooth in the linear system $|H|$ defined by a hyperplane section and contains the points $p,p'$. But now by definition of $T$, $4p+q+r=X\cdot  T_pX$ and $4p'+q+r= X\cdot T_{p'}X$. Thus by the same argument used in the proof of Lemma \ref{lem:clueT}, $4p+q+r$ and $4p'+q+r$ are rationally equivalent on $X$. This means that $T_pX$ and $T_{p'}X$ intersect in the line through $r$ and $q$ and define a pencil which is a $\bP^3$ cutting $X$ on a curve. The pencil restricted to such a curve contains both divisors $4p+q+r$ and $4p'+q+r$. This curve is in fact $C=H_{p,p'}^{q,r}\cap X$, so $4p+q+r$ and $4p'+q+r$ are in particular rationally equivalent on it. 
By difference, $4p$ and $4p'$ and rationally equivalent on $C$, which means $[4p-4q]=0\in JC$.

This defines a map $U\to Z'_4(H),$
from an open dense subset $U\subseteq T^0\times_{\pi_{23,T^0}(T^0) }T^0$.
The map sends a point $((p,q,r), (p',q,r))\in U$ to the point $(p,p',C)\in Z'_4(H)$, where $C$ defined as above. 

We now prove that, up to shrinking $U$, the rational map is finite. For this, its enough to notice that a point $(p,p',C)\in Z^{'0}_4(H)$ in the image is of the form $(\pi_1(p',q,r), \pi_1(p',q,r), C=\left \langle T_pX, T_{p'}X\right\rangle\cap X)$ and by Theorem \ref{thm:corrT} the fiber $\pi_{23}((r,s))$ is generically finite. 

We now prove that the rational map is surjective. Let $(p,p',C)\in Z^{'0}_4(H)$, then the condition $4[p-p']=0\in JC$ define a $g_1^4$. All $g_1^4$ in a complete intersection in $\bP^4$ are given by projections from lines defined by the two residual points $q,r$. More precisely, for $p,p'\in C$ satisfying $4[p-p']=0\in JC$, let $H_p, H_p'$ be the hyperplanes through $p,p'$ (respectively) and containing the line through $q,r$. Then $H_p\cdot C=4p+q+r$ and $H_{p'}\cdot C=4p'+q+r$. Thus, we find a point $((p,q,r)(p',q,r))\in T^0\times_{\pi_{23,T^0}(T^0) }T^0$. We conclude from above that $Z'^0_4(H)$, and so also $Z_4(H)$, are non-empty of dimension $2$. 
%In particular, $Z_4(H)$ is also not empty and we now prove that it has the same dimension. 
%For this, consider the forgetful map $Z_4'^0(H)\to Z_4^0(H)$ sending $(p,p',C')$ to $(p,p')$. We need to prove that there is not a positive dimensional locus of curves $C'$ smooth through a general choice of $p,p'$ such that $4[p-q]=0\in JC'$. Assume by contradiction that this is the case. The rational map constructed above gives on the image for any $p,p'$ a precise curve $C'$ and a point $q\in C'$. The points $q_i$ for fixed $p,p'$ are finite and $C'$ is uniquely recovered from them. Thus we have a two dimensional family $(p_t,p'_t)$ of pairs of points of $X$ defining an open subset of points $(C_t,p_t,q_t)$ of the image of the rational map in $Z'_4(L)$. Consequently, if for a general choice of $(p,p')$ in this image we had a positive dimensional locus of curves $C'$, then $Z'_4(L)$ would have dimension bigger than the one computed. Since we have prove that the map. This gives a contradiction.
\end{proof}

%\subsection{The correspondence $T$ and constant cycle curves}

%The problem is that it is not easy to compute the degree of any projection but the first one...

%\section{A correspondence $M$ for points of multiplicity $5$ (\rtext{6?}) of a complete intersection of three quadric in $\bP^5$}\label{sec:M}
%In this section we introduce a correspondence $M$ on a smooth complete intersection $X\subseteq \bP^5$ of three quadrics, detecting $3$-planes whose intersection with $X$ contains a point of at least multiplicity $5$ (\rtext{6?}).  

%\begin{theorem}\label{thm:corrM} Let $X\subseteq \bP^5$ be a smooth complete intersection of three quadrics, everything defined over $\bC$. Denote by $H^p_{q,r,s}$ the $3$-plane through four distinct points $p,q,r,s$ of $X$. Then the closure $M\subseteq X\times X\times X\times X$ of 
%\begin{equation}
 %   M^0= \{(p,q,r)\in X\times X\times X\,|\, X\cdot H^p_{q,r,s}=5p+q+r+s \in Z_0(X)\}
%\end{equation}
%........to do  
%\end{theorem}

%rtext{The point here is that $X\cdot H^p_{q,r,s}\supset 5p$, then $H^p_{q,r,s} \supset T_pX$. The plane $T_pX$ should cut a surface on any quadric singular at $p$. So roughly speaking I should get a contribution of $2$ from each quadric and get at least $6$. But it might be wrong (hopefully)..}

\end{document}